\documentclass[a4paper,12pt,intlimits,oneside]{amsart}
\usepackage{graphicx}

\usepackage{enumerate}
\usepackage{amsfonts}
\usepackage{amsmath,amsthm,amssymb}

\newcommand{\comment}[1]{}

\newtheorem{theorem}{Theorem}

\newtheorem{proposition}[theorem]{Proposition}
\newtheorem{lemma}[theorem]{Lemma}
\newtheorem{corollary}[theorem]{Corollary}

\newtheorem{remark}[theorem]{Remark}
\newtheorem{conjecture}[theorem]{Conjecture}

\newcommand\Var{\qopname\relax o{Var}}

\newcommand\de{\delta}

\newcommand\f{\varphi}

\newcommand{\NN}{\mathbb N}
\newcommand{\ZZ}{\mathbb Z}
\newcommand{\RR}{\mathbb R}
\newcommand{\CC}{\mathbb C}
\newcommand{\TT}{\mathbb T}

\newcounter{rek}
\newcounter{rev}
\newcommand{\rev}[1]{
    \marginpar{\refstepcounter{rev}{\small\therev(Sz.R.): #1}}
}

\newcounter{rem}
\newcounter{rec}
\begin{document}

\title[Hardy-Littlewood majorant problem]{Special quadrature error estimates and their application in the Hardy-Littlewood majorant problem}\thanks{Supported in part by the Hungarian National Foundation for Scientific Research, Project \#'s K-81658 and K-100461.}

\author{S\'andor Krenedits}

\date{\today}



\begin{abstract} The Hardy-Littlewood majorant problem has a positive answer only for exponents $p$ which are even integers, while there are counterexamples for all $p\notin 2\NN$. Montgomery conjectured that there exist counterexamples even among idempotent polynomials. This was proved recently by Mockenhaupt and Schlag with some four-term idempotents.

However, Mockenhaupt conjectured that even the classical $1+e^{2\pi i x} \pm e^{2\pi i (k+2)x}$ three-term character sums, should work for all $2k<p<2k+2$ and for all $k\in \NN$. In two previous papers we proved this conjecture for $k=0,1, 2,3,4$, i.e. in the range $0<p<10$, $p\notin 2\NN$. Here we demonstrate that even the $k=5$ case holds true.

Refinements in the technical features of our approach include use of total variation and integral mean estimates in error bounds for a certain fourth order quadrature. Our estimates make good use of the special forms of functions we encounter: linear combinations of powers and powers of logarithms of absolute value squares of trigonometric polynomials of given degree. Thus the quadrature error estimates are less general, but we can find better constants which are of practical use for us.
\end{abstract}

\maketitle

\vskip1em \noindent{\small \textbf{Mathematics Subject
Classification (2000):} Primary 42A05. \\[1em]
}



\section{Introduction}\label{sec:intro}

Let $\TT:=\RR/\ZZ$. The Hardy-Littlewood majorization problem \cite{HL} is the question if for any pair of functions $f,g :\TT \to \CC$ with $|\widehat{g}|\leq \widehat{f}$ -- that is, with $f$ majorizing $g$ -- do we necessarily have $\|g\|_p\leq \|f\|_p$?

Hardy and Littlewood noted that the Parseval identity easily implies this for all $p\in 2\NN$ an even integer, but they also found that for $p=3$ the property fails. Indeed, they took $f=1+e_1+e_3$ and $g=1-e_1+e_3$ (where $e_k(x):=e(kx)$ and $e(t):=e^{2\pi i t}$) and calculated that $\|f\|_3<\|g\|_3$. Later counterexamples were found by Boas \cite{Boas} for all $p\ne 2k$ and Bachelis \cite{Bac} showed that not even allowing a constant factor $C_p$ (i.e. requiring only $\|g\|_p\leq C_p \|f\|_p$) could save the property.

Montgomery conjectured that the majorant property for $p\notin 2\NN$ fails also if we restrict to \emph{idempotent} majorants, see \cite[p. 144]{Mon}. (An integrable function is idempotent if its convolution square is itself: that is, if its Fourier coefficients are either 0 or 1.) This has been recently proved by Mockenhaupt and Schlag in \cite{MS}. Their example is a four-term idempotent $f$ and a signed version of it for $g$. For more details and explanations of methods and results see \cite{Krenci,Krenci2} and the references therein.

In this paper we will be concerned with the even sharper conjecture, suggested by Mockenhoupt in his habilitation thesis \cite{Moc}.

\begin{conjecture}\label{conj:con3} Let $2k<p<2k+2$, where $k\in \NN$ arbitrary.
Then the three-term idempotent polynomial $P_k:=1+e_1+e_{k+2}$ has
smaller $p$-norm than $Q_k:=1+e_1-e_{k+2}$.
\end{conjecture}

Mockenhoupt presented an incomplete argument for the $k=1$ case already in \cite{Moc}. His argument hinted that some numerical analysis may be used in the proof, but we could not complete the solution along those lines. Nevertheless, we have proved this conjecture for $k=0,1,2$ in \cite{Krenci} and later even to $k=3,4$ in \cite{Krenci2}.

One motivation for us was the recent paper of Bonami and R\'ev\'esz \cite{AJM}, who used suitable idempotent polynomials as the base of their construction, via Riesz kernels, of highly concentrated ones in $L^p(\TT)$ for any $p>0$. These key idempotents of Bonami and R\'ev\'esz had special properties, related closely to the Hardy-Littlewood majorant problem. For details we refer to \cite{AJM}. For the history and relevance of this closely related problem of idempotent polynomial concentration in $L^p$ see \cite{AJM, L1conc}, the detailed introduction of \cite{Krenci}, the survey paper \cite{Ash2}, and the references therein. The Bonami-R\'ev\'esz construction, after suitable modification, directly and analytically gave the result for $k=0$.

For larger $k$, however, in \cite{Krenci, Krenci2} we used function calculus and support our analysis by numerical integration and error estimates where necessary. Naturally, these methods are getting computationally more and more involved when $k$ is getting larger. "Brute force" numerical calculations still lead to convincing tables and graphes, but the increase of the number of nodes in any quadrature formula endanger the prevalence of theoretical error bounds due to the additional computational error, however small for reasonably controlled step numbers, but possibly accumulating for very large step numbers.

Striving for a worst-case error bound incorporating also the computational error, we thus settled with the goal of keeping any numerical integration, i.e quadrature, under the step number $N=500$, that is step size $h=0.001$. Calculation of trigonometrical and exponential functions, as well as powers and logarithms, when within the numerical stability range of these functions (that is, when the variables of taking negative powers or logarithms is well separated from zero) are done by mathematical function subroutines of usual Microsoft Excel spreadsheet, which computes the mathematical functions with 15 significant digits of precision. Although we do not detail the estimates of the computational error of applying spreadsheets and functions from Microsoft Excel tables, it is clear that under this step number size our calculations are reliable well within the error bounds. For a more detailed error analysis of that sort, which similarly applies here, too, see our previous work \cite{Krenci}, in particular footnote 3 on page 141 and the discussion around formula (22), and see also the comments in the introduction of \cite{Krenci2}.

We keep using the fourth order quadrature formula, presented and explained in \cite{Krenci2}, see \cite[Lemma 5]{Krenci2}. However, another new argument also has to be invoked for $k=5$ compared to $k=3, 4$, because in this case the analytic scheme of proving fixed signs of certain derivatives simply break down. Using the special form of our integrands and the resulting form of estimates with the initial trigonometrical functions, we thus invoke the special quadrature error estimate of Lemma \ref{l:superquadrature}. These estimates make good use of the concrete form, local maximum values and alike, of the functions $G^t$ in question, but the theoretical estimates with $\Var G$  (the total variation of the function $G$) and $\sum G^t(\zeta)$ over local maximum values $\zeta$ of $G$, might have some theoretical interest, too.

Second, as already suggested in the conclusion of \cite{Krenci} and applied in \cite{Krenci2} for $k=4$, we use Taylor series expansion at more points than just at the midpoint $t_0:=k+1/2$ of the $t$-interval $(k,k+1)$, thus reducing the size of powers of $(t-t_0)$, from powers of $1/2$ to powers of smaller radii.

Finally, we needed a further consideration in proving that the approximate Taylor polynomial $P(t)$, minus the allowed worst case error $\delta$, still stays positive in the interval of our Taylor expansion. Basically, in \cite{Krenci} we could always use that the polynomials $p(t):=P(t)-\delta$ were totally monotone -- now some occurring approximate Taylor polynomials will not have this feature, and we need a more refined calculus to succeed in proving their constant sign over the interval of investigation.

Key to this is the consideration of the variance of some of the derivatives of $p$, for if a function vanishes somewhere inside an interval, than its variance exceeds the sum of the absolute values taken at the left and right endpoints of the interval considered. This elementary fact comes to our help in concluding that $p(t)$, and hence the considered difference function $d(t)$, approximated by $P(t)$ within a certain error $\de$, keeps constant size; for if the first $j$ derivatives are positive at the left endpoint, and the $j$th derivative preserves the positive sign all over the whole interval, then there is no way for $p$ to vanish anywhere in the interval.

\section{Boundary cases of Conjecture \ref{conj:con3} at $p=2k$ and $p=2k+2$}\label{sec:endpoints}

Let $k\in \NN$ be fixed. (Actually, later we will work with $k=5$ only.)

We now write $F_{\pm}(x):=1+ e(x)\pm e((k+2)x)$ and consider the $p^{\rm th}$ power integrals $f_{\pm}(p):=\int_0^1 |F_{\pm}(x)|^p dx$ as well as their difference
$
\Delta(p):=f_{-}(p)-f_{+}(p):=\int_0^1 |F_{-}(x)|^p-\int_0^1 |F_{+}(x)|^p dx.
$
Our goal is to prove Conjecture \ref{conj:con3}, that is $\Delta(p)>0$ for all $p\in(2k,2k+2)$.

Let us introduce a few further notations. We will write $t:=p/2\in[k,k+1]$ and put
\begin{align}\label{eq:Gpmdef}
G_{\pm}(x)&:=|F_{\pm}(x)|^2,\qquad
g_{\pm}(t):=\frac12 f_{\pm}(2t)= \int_0^{1/2} G_{\pm}^t(x) dx,\qquad \\ \label{eq:ddef} d(t)&:=\frac 12 \Delta(2t)=g_{-}(t)-g_+(t)=\int_0^{1/2} \left[ G_{-}^t(x)-G_{+}^t(x)\right] dx.
\end{align}
Formula \eqref{eq:ddef} also yields that denoting $H_{t,j,\pm}(x):=G_{\pm}^t(x)\log^{j}G_{\pm}(x)$ the explicit integral formula
\begin{align}\label{eq:djdef}
d^{(j)}(t)=g_{-}^{(j)}(t)-g^{(j)}_+(t)& =\int_0^{1/2} G_{-}^t(x)\log^{j}G_{-}(x)dx -\int_0^{1/2} G_{+}^t(x)\log^{j}G_{+}(x) dx \notag \\ & =\int_0^{1/2} H_{t,j,-}(x)dx -\int_0^{1/2} H_{t,j,+}(x)dx .
\end{align}
holds true, and so in particular
\begin{equation}\label{eq:djnormbyHL1}
|d^{(j)}(t)| \leq \|H_{t,j,+}\|_{L^1[0,1/2]} + \|H_{t,j,-}\|_{L^1[0,1/2]} \qquad (j\in \NN).
\end{equation}

We are to prove that $d(t)>0$ for $k<t<k+1$. First we show at the endpoints $d$ vanishes; and, for later use, we also compute some higher order integrals of $G_{\pm}$. Actually, here we can make use of the following lemma, already proven in \cite[Lemma 3]{Krenci2}.
\begin{lemma}\label{l:GParseval} Let $\rho\in \NN$ with $1 \leq \rho \leq k+1$. Then we have
\begin{equation}\label{eq:Gpmadelldeveloped}
G_{\pm}^{\rho}=|F_{\pm}^{\rho}|^2 =\left| \sum_{\nu=0}^{\rho\cdot (k+2)} a_{\pm}(\nu) e_{\nu} \right|^2 \qquad {\rm with}\qquad
a_{\pm}(\nu):=(\pm 1)^\mu \binom{\rho}{\mu}\binom{\rho-\mu}{\lambda},
\end{equation}
where $\mu:=\left[\dfrac{\nu}{k+2}\right]$ and $\lambda:=\nu-\mu(k+2)$ is the reduced residue of $\nu \mod k+2$. Therefore,
\begin{equation}\label{eq:Gtorhointegral}
\int_0^{1/2} |G_{\pm}|^{\rho} =\frac12 \sum_{\nu=0}^{\rho\cdot(k+2)} \left| a_{\pm}(\nu) \right|^2. \end{equation}
In particular, $\int_0^{1/2} |G_{+}|^{\rho}=\int_0^{1/2} |G_{-}|^{\rho}$ for all $0 \leq \rho\leq k+1$ and thus $d(k)=d(k+1)=0$.
\end{lemma}

Apart from the immediate result that $d$ vanishes at the endpoints of the critical interval $[k,k+1]$, we will make further use of the above explicit computation of $\rho^{\rm th}$ power integrals of $G$. To that we need the precise values of these square sums of coefficients, which is easy to bring into a more suitable form for direct calculation. Namely we have
\begin{align}\label{eq:Arho}
A(\rho)& :=\sum_{{\nu=0\atop \mu:=\left[\frac{\nu}{k+2}\right] }\atop \lambda:=\nu-\mu(k+2)}^{\rho\cdot(k+2)} \binom{\rho}{\mu}^2 \binom{\rho-\mu}{\lambda}^2 = \sum_{\mu=0}^{\rho} \binom{\rho}{\mu}^2 ~\sum_{\lambda=0}^{\rho-\mu} \binom{\rho-\mu}{\lambda}^2 = \sum_{\mu=0}^{\rho} \binom{\rho}{\mu}^2 \binom{2\rho-2\mu}{\rho-\mu} \notag \\
& =1,~ 3,~ 15,~ 93,~ 639,~ 4653,~35169  \quad {\rm for} \quad \rho=0,1,2,3,4,5,6, \quad {\rm respectively}.
\end{align}
\begin{corollary}\label{cor:Parseval} For all $\rho\leq (k+1)$ we have $\int_0^{1/2} G_{\pm}^{\rho}= \frac12 A(\rho)$ with the constants $A(\rho)$ in \eqref{eq:Arho}.
\end{corollary}
With the aid of these explicit values, even arbitrary power integrals of $G_{\pm}$ can be estimated.
\begin{proposition}\label{prop:powerintegralsofG} Let $\rho \in \NN$ and $\rho\leq k+1$. Then with the constants $A(\rho)$ in \eqref{eq:Arho} we have
\begin{equation}\label{eq:GtotauMaxPars}
\int_0^{1/2} G_{\pm}^{\tau} \leq \frac12 9^{\tau-\rho} A(\rho) \quad (\tau<\rho)\qquad {\rm and} \qquad
\int_0^{1/2} G_{\pm}^{\tau} \leq \frac12 A^{\tau/\rho}(\rho)\quad (\tau>\rho).
\end{equation}
\end{proposition}
\begin{proof} As $0\leq G\leq 9$, for the first estimate one can use $G^\tau\leq 9^{\tau-\rho}G^\rho$. The second estimate is directly furnished by H\"older's inequality with exponents $p=\rho/\tau>1$ and $q=1-1/p$.
\end{proof}

\section{Analysis of $G_{\pm}$}\label{sec:G}

To start the analysis of $G(x):=G_{5,\pm}(x)$, let us compute its $x$-derivatives. As in formula (7) and the following lines of \cite{Krenci2}, in case $k=5$ we find easily
\begin{align}\label{eq:GpmdefFirst}
G_{\pm}(x)&=3+2\{\cos(2\pi x)\pm \cos(12\pi x)\pm \cos(14\pi x)\}\\
G_{\pm}^{(2m+1)}(x)&=(-4)^{m+1}\pi^{2m+1} \cdot \left\{ \sin(2\pi x)\pm 6^{2m+1}\sin(12\pi x) \pm 7^{2m+1} \sin(14\pi x) \right\}, \notag\\
G_{\pm}^{(2m)}(x)&=2(-4)^{m}\pi^{2m} \cdot \left\{ \cos(2\pi x)\pm 6^{2m}\cos(12\pi x)\pm 7^{2m}\cos(14\pi x)\right\}.\notag
\end{align}
Consequently we have
\begin{equation}\label{eq:Mmk=5}
\|G_{\pm}\|_\infty \leq 9=:M_0, \qquad \| G_{\pm}^{(m)}\|_\infty \leq 2^{m+1} \pi^{m} \{1+6^{m}+ 7^{m}\}=:M_m \qquad (m=1,2,\dots),
\end{equation}
that is $\| G_{\pm}^{(m)}\|_\infty \leq M_m$ ($m=0,1,2,3,4$) with
\begin{align}\label{eq:Gpmadmnorm}
M_0&=9, \qquad \qquad M_1=175.929... <176, \qquad \qquad  M_2=6790.287... <6800,~\\ M_3&=277, 816.239... < 280,000, \qquad \qquad M_4=11 ~ 527,002.2... < 11,600,000. \notag
\end{align}

\begin{lemma}\label{l:Gmaxima} Both functions $G_+(x)$ and $G_-(x)$ have seven local maxima in $\TT\equiv (-\frac12, \frac12]$. At zero-symmetric pairs of maximum places of these even functions the same maximum values occur, so presenting these values with multiplicity 2, they are the following:
\end{lemma}

\begin{tabular}{|c|c|c||c|c|c|}
\hline
\multicolumn{3}{|c||}{$G_+(x)$} & \multicolumn{3}{|c|}{$G_-(x)$} \\
\hline
$\zeta$ &  $G_+(\zeta)$ &  multiplicity & $\zeta$ &  $G_-(\zeta)$ & multiplicity \\
\hline\hline
$0$ & $9$ & $1$ & $\approx \pm 0.076$ & $< 8.662$ & $2$ \\
$\approx \pm 0.151$ & $< 7.701$ & $2$ & $\approx \pm 0.227$ & $< 6.279$ & $2$ \\
$\approx \pm 0.302$ & $< 4.628$ & $2$ & $\approx \pm 0.377$ & $< 3.005$ & $2$ \\
$\approx \pm  0.448$ & $< 1.661$ & $2$ & $0.5$ & $1$ & $1$ \\
\hline
\end{tabular}

\medskip

\begin{proof}
As $G$ is a degree 7 trigonometric polynomial, if it has $n$ local maximums, then there are the same number of interlacing minimums, so altogether $2n\leq 2\deg G' = 14$ roots of $G'$. So the number of maxima is at most $7$, which we will find -- taking into account evenness of $G$, and thus the same symmetrically located  maxima and minima in $[-1/2,0]$ and in $[0,1/2]$ -- so no further local maxima can exists.

As $G$ is even, it suffices to analyze $[0,1/2]$. We start examining the functions by tabulating it with step size $h=0.001$, and identifying the indices $i$ where monotonicity of the $G(x_i)$ turns from increase to decrease. Then there has to be a local maximum at some point $\zeta_i\in [x_{i-1},x_{i+1}]$. Clearly one of the nodes $x_i'\in \{x_{i-1},x_i,x_{i+1}\}$ has $|\zeta_i-x'_i|\leq h/2$. The second order Taylor expansion around $\zeta_i$ now gives $G(\zeta_{i})-G(x'_{i})\le \frac12 ||G''||_{\infty} (\frac{h}2)^2$, as $G'(\zeta_i)=0$. Applying \eqref{eq:Mmk=5} $M_2<6800$ and $h=0.001$, we obtain $G(\zeta_{i})-G(x'_{i})\le 0.00085 <\delta:=0.001$.

In the table above we recorded $\zeta_i\approx x_i$ with error $<h=10^{-3}$ and an upper estimation of the corresponding maxima using $G(\zeta_i)< G(x'_{i})+\delta \leq \max \big(G(x_{i-1}),G(x_i),G(x_{i+1})\big)+\delta$.
\end{proof}

Denote $\Var(\psi,[a,b])$ the total variation of the function $\psi$ on $[a,b]$, and in particular let $\Var(\psi):=\Var(\psi,\TT)$. As an immediate corollary to the above lemma, we formulate here
\begin{corollary}\label{cor:VarG} Denote by $Z:=Z_{\pm}$ the set of local maximum points of $G=G_{\pm}$. For any positive parameter $t>0$ we have $\Var(G^t) < 2 \sum_{\zeta \in Z} G^t(\zeta)$. In particular, $\Var (G_{\pm})<74$.
\end{corollary}
\begin{proof}
It is easy to see that for a piecewise monotonic function $\psi$ one has $\Var(\psi,[a,b])=\Var(|\psi|,[a,b])$. It follows that $\Var(|\psi|,[a,b])=\int_a^b |\psi'|$.

Furthermore, for a piecewise monotonic function, like $G$ or $G^t$, the total variation is the sum of the change of the function on each of its monotonicity intervals. Since we are talking about periodic functions, i.e. functions on $\TT$, with only finitely many critical points, it is clear that the local maximum and minimum places -- with the latter denoted by $\Omega\subset \TT$, say -- interlace and monotonicity segments connect these neighboring local extremum places. Therefore the total sum of all the changes is
$$
\Var(G^t)= \sum_{\theta,\eta\in Z\cup\Omega \atop (\theta,\eta)\cap(Z\cup \Omega)=\emptyset} |G^t(\theta)-G^t(\eta)| = 2\sum_{\zeta\in Z} G^t(\zeta)-  2 \sum_{\omega\in \Omega} G^t(\omega) < 2 \sum_{\zeta\in Z} G^t(\zeta),
$$
taking into account $G\ge0$, too.

Whence the first assertion of the Corollary, while the last is just a small calculation adding the maxima (taken into account according to multiplicity) in the columns of the table of maxima in Lemma \ref{l:Gmaxima}.
\end{proof}

\section{Estimates of $|H(x)|$ and of $\|H^{IV}\|_\infty$}\label{sec:H}

Let us start analyzing the functions
\begin{equation}\label{eq:Hdef}
H(x):=H_{t,j,\pm}(x):=G^t(x)\log^j G(x)\qquad (x\in [0,1/2]) \qquad (t\in [k,k+1],\, j\in \NN).
\end{equation}
To find the maximum norm of $H_{t,j,\pm}$, we in fact look for the maximum of an expression of the form $v^t |\log v|^j$, where $v=G(x)$ ranges from zero (or, if $G \ne 0$, from some positive lower bound) up to $\|G\|_\infty\leq 9$.

A direct calculus provides a description of the behavior of the function $\alpha(v):=\alpha_{s,m}(v):= v^s| \log v|^m$ for any $s>0$ and $m\in\NN$ on any finite interval $[a,b]\subset [0,\infty)$, see \cite[Lemma 6]{Krenci2}. 

For the application of the above quadrature \eqref{eq:quadrature} we calculated (c.f. also \cite[(15)]{Krenci2})
\begin{align}\label{eq:Hdoubleprimegeneral}
H''(x)&:=H''_{t,j,\pm}(x) = G''(x) G^{t-1}(x) \log^{j-1} G(x) \left\{t \log G(x)+j \right\} \notag \\ & ~~+ G'^2(x) G^{t-2}(x) \log^{j-2} G(x) \left\{ t(t-1) \log^2 G(x)  + j(2t-1) \log G(x) + j(j-1) \right\}.
\end{align}
However, the error estimation in the above explained quadrature approach forces us to consider even fourth $x$-derivatives of $H=H_{t,j,\pm}$ using
$H^{IV}= \sum_{m=0}^4 \binom{4}{m} (G^t)^{(m)} (\log^j G)^{(4-m)}$.
We have already computed in \cite{Krenci2} respective formulae
for $(G^{t})^{(m)}$ and $(\log G^j)^{(m)}$ for $m=1,2,3,4$ (c.f. \cite[(17), (18)]{Krenci2}). Substituting these in $H^{IV}$ resulted in the general formula \cite[(19)]{Krenci2} stating with $L:=\log G$
\begin{align}\label{eq:HIVgeneralformula}
H^{IV}&=G^{t-4}G'^4\Big\{ j(j-1)(j-2)(j-3)L^{j-4}
\notag \\ & \qquad + [4t-6]j(j-1)(j-2)L^{j-3} +[6t^2-18t+11]j(j-1)L^{j-2}
\notag \\ & \qquad +[2t^3-9t^2+11t-3]2jL^{j-1} + t(t-1)(t-2)(t-3) L^j\Big\}
\notag \\ & + 6\cdot G^{t-3}G'^2G''\Big\{ j(j-1)(j-2)L^{j-3}+3(t-1)j(j-1)L^{j-2}
\notag \\ & \qquad +[3t^2-6t+2)]j L^{j-1} +t(t-1)(t-2)L^j \Big\}
+ G^{t-1}G^{IV}\left\{jL^{j-1}+tL^j \right\}
\\ & \Big(3\cdot G^{t-2}G''^2 + 4\cdot G^{t-2}G'G'''\Big) \Big\{ j(j-1)L^{j-2}+(2t-1)jL^{j-1}+t(t-1)L^j\Big\}\notag .
\end{align}

Finally, from that and writing in $\|G^{(m)}\|_\infty \leq M_m$
we were led to the general estimate \cite[(20)]{Krenci2}.
As now the values of $M_m$ are estimated by \eqref{eq:Gpmadmnorm}, the corresponding values can be written in, and putting also $ \ell:=|L|=|\log G|$ formula \cite[(20)]{Krenci2} yields
\begin{align}\label{eq:HIVgeneralestimatewMmin}
|H^{IV}|&\leq 959,512,576 \cdot v^{t-4} \Big\{ j(j-1)(j-2)(j-3)\ell^{j-4} \notag
\\ & \qquad + [4t-6]j(j-1)(j-2)\ell^{j-3} +[6t^2-18t+11]j(j-1)\ell^{j-2} \notag
\\ & \qquad +[2t^3-9t^2+11t-3]2j\ell^{j-1} + t(t-1)(t-2)(t-3) \ell^j\Big\} \notag
\\ & + 1,263,820,800\cdot v^{t-3}\Big\{j(j-1)(j-2)\ell^{j-3}+3(t-1)j(j-1)\ell^{j-2}
\\ &\qquad +[3t^2-6t+2]j \ell^{j-1} +t(t-1)(t-2)\ell^j \Big\}
+  11,600,000 \cdot v^{t-1}\left\{j\ell^{j-1}+t\ell^j \right\} \notag
\\ & +  335,840,000 \cdot v^{t-2}
\Big\{j(j-1)\ell^{j-2}+ (2t-1)j\ell^{j-1}+t(t-1)\ell^j \Big\} \notag .
\end{align}

On the other hand, for reasons becoming apparent only later from the improved quadrature error estimate in Section \ref{s:quadrature}, here we need to derive another consequence of formula \eqref{eq:HIVgeneralformula}. We now substitute the norm estimates of \eqref{eq:Gpmadmnorm} by $M_m$'s into \eqref{eq:HIVgeneralformula} only partially, that is, we leave (apart from all powers of $G$) even one of $G'$ without estimation by $M_1$, wherever $G'$ occurs, in order to take advantage of our quadrature utilizing expressions of the form $G^t|G' \log^j G|$.
Inserting $k=5$ and the numerical values of $M_1,M_2,M_3$ and $M_4$ from \eqref{eq:Gpmadmnorm} this leads to
\begin{align}\label{eq:HIVgeneralrafkos}
|H^{IV}|&\leq 5,451,776 \cdot G^{t-4}|G'|
\Big\{ j(j-1)(j-2)(j-3)\ell^{j-4}
\notag \\ & \quad \qquad + [4t-6]j(j-1)(j-2)\ell^{j-3} +[6t^2-18t+11]j(j-1)\ell^{j-2}
\notag \\ & \quad \qquad +[2t^3-9t^2+11t-3]2j\ell^{j-1} + t(t-1)(t-2)(t-3) \ell^j\Big\}
\notag \\& + 7,180,800 \cdot G^{t-3}|G'|\Big\{ j(j-1)(j-2)\ell^{j-3}+3(t-1)j(j-1)\ell^{j-2}
\notag \\&\quad \qquad +[3t^2-6t+2)]j \ell^{j-1} +t(t-1)(t-2)\ell^j \Big\}
\\& + 1,120,000 \cdot G^{t-2}|G'|\Big\{ j(j-1)\ell^{j-2}+(2t-1)j\ell^{j-1}+t(t-1)\ell^j\Big\} \notag \\& +138,720,000\cdot G^{t-2}\Big\{j(j-1)\ell^{j-2}+(2t-1)j\ell^{j-1}+t(t-1)\ell^j \Big\}
\notag \\&  + 11,600,000 \cdot G^{t-1}\left\{j\ell^{j-1}+t\ell^j \right\} \qquad \qquad \left( \textrm{with, as always,} \quad \ell:=|L|=|\log G| \right). \notag
\end{align}

\section{Quadrature with variation}\label{s:quadrature}

In the paper \cite{Krenci} we used Riemann sums when numerically integrating the functions $H:=G^t\log^j G$ along the $x$ values. A new feature of the subsequent paper \cite{Krenci2}, among other things, was the application of a higher order quadrature formula- Namely, in \cite{Krenci2}, formula (12) and (13) we recalled the following easy-to prove elementary fact. Let $\f$ be a four times continuously differentiable function on $[0,1/2]$, $N\in \NN$, $h:=1/(2N)$ and denote $x_n:=\dfrac{2n-1}{4N}$ for $n=1,2,\dots,N$. Then we have
\begin{equation}\label{eq:prequadrature}
\left| \int_0^{1/2} \f (x) dx - \sum_{n=1}^N \left\{ \frac{1}{2N} \f\left(x_n\right)  + \frac{1}{192N^3} \f'' \left(x_n\right) \right\} \right| \leq \frac{1}{60\cdot 2^{10} N^5} \sum_{n=1}^N \max_{|x-x_n|\leq \frac{h}{2}} |\f^{IV}(x)| .
\end{equation}
In \cite{Krenci2} we then used this with the further obvious estimate $\max_{|x-x_n|\leq \frac{h}{2}} |\f^{IV}(x)|\leq \|\f^{IV}\|_{\infty}$, resulting in
the further estimation of \eqref{eq:prequadrature} by $\frac{\|\f^{IV}\|_{\infty}}{60 \cdot 2^{10} N^4}$.
\begin{equation}\label{eq:quadrature}
\left| \int_0^{1/2} \f (x)dx - \sum_{n=1}^N \left\{ \f\left(\frac{2n-1}{4N}\right) \frac{1}{2N} + \f'' \left(\frac{2n-1}{4N}\right) \frac{1}{192N^3} \right\} \right| \leq \frac{\|\f^{IV}\|_{\infty}}{60 \cdot 2^{10} N^4} .
\end{equation}

We intend to use the quadrature formula \eqref{eq:quadrature} to compute approximate values of $d'(5), d''(5)$, $d'''(5)$ and then even $d^{(j)}(t_0)$ with various values of $t_0\in [5,6]$ and $j\in \NN$. However, use of direct estimations of $\|H_{t_0,j,\pm}^{IV}\|_\infty$ in the quadrature would result in step numbers as high as 800, already inconveniently large for our purposes. Thus here we invoke a further, more detailed analysis of the quadrature formula, aiming at bounding the step number further down below 500 with the improved error estimation.

The basic idea is that we try to apply \eqref{eq:prequadrature} directly. For continuous $\f^{IV}$, the local maximum are attained at certain points $\xi_n\in [x_n-h/2,x_n+h/2]$, and the error bound becomes $\sum_{n=1}^N |\f^{IV}(\xi_n)|$. In fact this sum is a Riemann approximate sum of the integral (and not the maximum) of the function $\f^{IV}$, so we will get approximately $2N \cdot \int_0^{1/2} |\f^{IV}|$. That is, we arrive at the $L^1$ norm, instead of the $L^{\infty}$ norm, of the function $\f^{IV}$.

So we try to make use of this observation for $H_{t_0,j,\pm}$ in place of $\f$. Again, direct estimation of the error in this approximation $\sum_{n=1}^N |H_{t_0,j,\pm}^{IV}(\xi_n)| \approx 2N \cdot \int_0^{1/2} |H_{t_0,j,\pm}^{IV}|$, even if theoretically possible, does not provide nice and numerically advantageous results. Instead, we estimate the function $H_{t_0,j,\pm}^{IV}$, similarly as above, with functions involving $G$, $\log G$ and even derivatives of $G$, and then split the estimation of the sum $\sum_{n=1}^N |H_{t_0,j,\pm}^{IV}(\xi_n)|$ to estimations of similar Riemann sums of such simpler functions. For such combinations as $G^t(x) \log^j G(x)$ or $G^t(x) \log^j G(x) G'(x)$, we will find suitable error bounds and explicit computations or estimations of the $L^1$-norms, finally resulting improved estimations of the error in the quadrature formula. More precisely, we can derive the following improved special quadrature estimation, which subsequently will be used for $H_{t_0, j,\pm}$ (with various $j$ and $t_0$) in place of $\f$.

\begin{lemma}\label{l:superquadrature} Let $B_r, D_r>0, 1\leq t_r \le T$ and $j_r\geq 0$ for $r=0,1,\dots,R$. Assume
\begin{equation}\label{eq:ficonditions}
|\f^{IV}(x)| \leq \sum_{r=0}^R \left\{B_r G^{t_r}(x)\cdot |\log G(x)|^{j_r} + D_{r} G^{t_r}(x)|G'(x)|\cdot |\log G(x)|^{j_r}\right\}.
\end{equation}
Then for arbitrary $N\in \NN$ the quadrature formula
\begin{align}\label{eq:superquadrature}
& \left| \int_0^{1/2} \f - \sum_{n=1}^N \left\{ \f\left(\frac{2n-1}{4N}\right) \frac{1}{2N} + \f'' \left(\frac{2n-1}{4N}\right) \frac{1}{192N^3} \right\} \right|
\notag \\ & \qquad \qquad \qquad \qquad \leq \frac{1}{60 \cdot 2^{10} N^5} \sum\limits_{r=1}^R \left\{ B_r Q_N(G,t_r,j_r) + D_r ~Q^{*}_N(G,t_r,j_r) \right\}
\end{align}
holds true with
\begin{align}\label{eq:Qtjplain}
Q_N(G,t,j):=\chi(j\ne 0) \left( \max_{[0,1/9]} v^t|\log v|^j \right) N + \log^j9 \left\{N \int_\TT G^t + \frac12 \Var(G^t)  \right\}
\end{align}
and
\begin{align}\label{eq:Qtjstar}\notag
Q^{*}_N(G,t,j):=\chi(j\ne 0) & \left( \max_{[0,1/9]} v^t|\log v|^j  \right) \left\{\frac{14}{9}N +1700 \right\} \\& + \log^j9 \left\{\frac{N}{t+1} \Var(G^{t+1}) + 88 \Var(G^t) + 1700 \sqrt{\int_\TT G^{2t}} \right\}.
\end{align}
\end{lemma}
\begin{proof} As in the preceding arguments, we denote $x_n:=(2n-1)/(4N)$ and $h:=1/(2N)$ for $n=1,\dots,N$ and even for $n=1-N,\dots,N$. By the inequality \eqref{eq:prequadrature}, the condition \eqref{eq:ficonditions} and making use that all the arising terms in this estimate are continuous and even, we find with some appropriate, symmetrically chosen $\xi_n\in [x_n-h/2,x_n+h/2]$ that
\begin{align}\label{eq:quadriple}
\notag  \bigg| \int_0^{1/2} \f  ~ & - \sum_{n=1}^N \left\{ \f(x_n) h + \f''(x_n)\dfrac{h^3}{24} \right\} \bigg| \leq \frac{h^5}{60 \cdot 2^5} \sum_{n=1}^N \max_{|x-x_n|\leq \frac{h}{2}}  |\f^{IV}(x)|
\\ \notag &
\leq \frac{h^5}{60 \cdot 2^5} \sum_{n=1}^N \max_{|x-x_n|\leq \frac{h}{2}} \left|   \sum_{r=0}^R \left\{B_r G^{t_r}(x) |\log G(x)|^{j_r} + D_{r} G^{t_r}(x)|G'(x)||\log G(x)|^{j_r}\right\} \right|
\\ & =
\frac{h^5}{60 \cdot 2^5} \frac12 \sum_{n=1-N}^N \sum_{r=0}^R \left\{B_r G^{t_r}(\xi_n) |\log G(\xi_n)|^{j_r} + D_{r} G^{t_r}(\xi_n)|G'(\xi_n)||\log G(\xi_n)|^{j_r}\right\}
\\ \notag & =
\frac{N^{-5} }{60\cdot 2^{11}}\sum_{r=0}^R \left\{B_r \sum_{n=1-N}^N G^{t_r}(\xi_n) |\log G(\xi_n)|^{j_r} + D_{r} \sum_{n=1-N}^N G^{t_r}(\xi_n)|G'(\xi_n)||\log G(\xi_n)|^{j_r} \right\}.
\end{align}
So we are left with the estimation of the inner sums. There are two type of sums here, the first being without $|G'(\xi_n)|$ and the second with its appearance. For a more concise notation let us introduce the exponent $\kappa\in\{0,1\}$, and then consider the generic inner sum
$S:=S(t,j,\kappa):=\sum_{n=1-N}^N G^{t}(\xi_n)|G'(\xi_n)|^\kappa |\log G(\xi_n)|^{j}$.

To start with, when $j=0$ we can directly compare this sum to the corresponding integral. Recall that for any function $\psi$ of bounded total variation $\Var(\psi):=\Var(\psi,[a,b])$ on an interval $[a,b]$, and for any partition of $[a,b]$ as $a=x_0<x_1<\dots<x_i<\dots<x_{M-1}<x_M=b$ with the fineness of the partition $\delta:=\max_{i=1,\dots,M} (x_{i}-x_{i-1})$ and with any selection of nodes $\theta_i\in [x_{i-1},x_i]$, the Riemann sum $\sum_{i=1}^M \psi(\theta_i)(x_{i}-x_{i-1})$ approximates $\int_a^b \psi(x)dx$ within the error $\de \Var(\psi)$. So we obtain
\begin{equation}\label{eq:S0evaluation}
S(t,0,\kappa)=\sum_{n=1-N}^N G^{t}(\xi_n)|G'(\xi_n)|^\kappa \leq 2N \int_\TT G^{t}|G'|^\kappa + 2Nh \Var(G^{t}|G'|^\kappa).
\end{equation}
If $\kappa=0$, then the first term is $2N\int_\TT G^t$, and for $\kappa=1$ it is nothing else than $2N\Var(\frac{1}{t+1}G^{t+1})$ on $\TT$. As $2Nh=1$, for $\kappa=0$ the second term is $\Var(G^t)$, while for $\kappa=1$ we can also obtain a similar type estimate using that $\Var(|\Psi|,[a,b])=\Var(\Psi,[a,b])= \int_a^b |\Psi'|$. Namely we obtain
\begin{align}\label{eq:VarGtG'}
\notag \Var(G^{t}|G'|)& =\int_\TT |(G^t G')'| \leq \int_\TT tG^{t-1} G'^2 + \int G^t |G''| \leq M_1 \int_\TT tG^{t-1} |G'| + \int G^t |G''| \\ & \leq 176 \Var(G^t) + \sqrt{\int_\TT G^{2t} \int_\TT G''^2 } \leq 176 \Var(G^t) +  3400 \sqrt{\int_\TT G^{2t}}
\end{align}
with an application of the Cauchy-Schwartz inequality and computing
$$
\sqrt{\int G''^2 } = 8\pi^2 \\ \cdot \sqrt{\frac{1+36^2+49^2}{2}}\approx 3395.144...<3400.
$$
So collecting terms furnishes
\begin{equation}\label{eq:Sjzerofinal}
S(t,0,\kappa)\leq \begin{cases} 2N \int_\TT G^t + \Var(G^t) &\textrm{if}~ \kappa=0,\\ 2N \frac{1}{t+1} \Var(G^{t+1}) + 176 \Var(G^t) +3400 \sqrt{\int_\TT G^{2t}} &\textrm{if}~ \kappa=1. \end{cases}
\end{equation}
Observe that the right hand side of this estimate is just $2 Q_N(G,t,0)$ and $2 Q_N^{*}(G,t,0)$ when $\kappa=0$ and 1, respectively, so the part of the assertion for $j=0$ is proved.

For $j>0$ we estimate $S(t,j,\kappa)$ by first cutting the sum into parts according to $\xi_n \in X:=\{ x\in\TT ~:~ 0\leq G(x)\leq 1/9 \}$ and $\xi_n \notin X$. The first of these sums can then be estimated by $\max_{0\leq v\leq 1/9} v^t|\log^jv| \cdot \sum_{\xi_n\in X} |G'(\xi_n)|^\kappa$, the sum being $\le$ constant $2N$ for $\kappa=0$ while for $\kappa=1$ approximately $2N \int_X |G'|=2N\int_\TT|G'|\chi_X$, where $\chi_X$ is the characteristic function of $X$.

That latter integral of $|G'|$ on $X$ is just the total variation of $G(t)$ along its segments of range between 0 and 1/9. More precisely, as $G$ is a trigonometric polynomial, hence piecewise smooth with at most (actually, exactly) $2\deg G= 14$ monotonicity intervals $I_m$ ($m=1,\dots,14$) within $\TT=\cup_{m=1}^{14} I_m$, this whole total variation can amount at most $14$ times the maximal possible variation from 0 to 1/9 on each part of $X$ belonging to one monotonic segment $I_m$. That is, $\int_X |G'| =\sum_{m=1}^{14} \int_{X\cap I_m} |G'| = \sum_{m=1}^{14} \Var(G,X\cap I_m) \leq 14 \cdot 1/9$. In all, the contribution of the main term $2N\int_X|G'|$ is at most $14/9\cdot 2N$. (In reality, that variation is numerically even less, but this term will not be too interesting anyway.)

Next we apply the general Riemann sum error estimate to $|G'|\chi_X$ to infer $\sum_{\xi_n\in X} |G'(\xi_n)|= \sum_{n=1-N}^{N} |G'(\xi_n)|\chi_X(\xi_n) \leq 2N \int_{\TT}|G'| \chi_X + 2N \cdot h\cdot \Var(|G'|\chi_X) \leq \dfrac{28}{9} N + \Var(G'\chi_X)$.

We now show that this latter variance does not exceed $\Var(G')$. In view of the additivity of the total variation on intervals, $\Var(G'\chi_X)=\sum_{m=1}^{14} \Var(G'\chi_X,I_m)$, so it suffices to prove $\Var(G'\chi_X,I_m) \leq \Var(G',I_m)$. Recall that $I_m=[a_m,b_m]$ is, by construction, one of the intervals of monotonicity of $G$, hence a segment of $\TT$ where $G'$ has constant sign, with zeroes (and sign changes) of $G'$ at both endpoints. Then either $G(a_m)$ is a local minimum of $G$ and $G(b_m)$ is a local maximum of it, or conversely, corresponding to the cases when on $I_m$ $G'\geq 0$ or $G'\leq 0$, respectively. By symmetry, we can restrict to the first case, when $G$ is increasing on $I_m$. If $G>1/9$ on $I_m$, that is, if already $G(a_m)>1/9$, then $I_m\cap X = \emptyset$ and $\Var(G'\chi_X,I_m)=0 < \Var(G',I_m)$. Also if $G\leq 1/9$ on the whole interval $I_m$, then $\Var(G'\chi_X,I_m)= \Var(G',I_m)$.  The only case when $I_m\cap X$ is nontrivial is when $I_m \cap X= [a_m,c_m]$ with $a_m<c_m<b_m$ and $G(a_m)<G(c_m)=1/9<G(b_m)$. In this case, however, $G'\chi_X=G'$ on $[a_m,c_m[$, has a jump from $G'(c_m)$ to $0$ at $c_m$, and constant zero afterwards until the end of the interval $I_m$, whence $\Var(G'\chi_X,I_m)= \Var(G',[a_m,c_m])+ |G'(c_m)-0|=\Var(G',[a_m,c_m])+ |G'(c_m)-G'(b_m)|\leq \Var(G',[a_m,c_m])+ \Var(G',[c_m,b_m]) =\Var(G',I_m)$. So indeed we have $\Var(|G'|\chi_X)\leq \Var(G')$. Finally, using the above estimation of $\sqrt{\int G''^2 }$,
$$
\Var(|G'|\chi_X)\leq \Var (G') =\int |G''| \leq \sqrt{\int G''^2 } <3400.
$$
Writing in the maximum of $v^t|\log v|^j$, collection of terms results in
\begin{equation}\label{eq:Xsumesti}
\sum_{\xi_n\in X} G^{t}(\xi_n)|G'(\xi_n)|^\kappa|\log G(\xi_n)|^{j} \leq \left( \max_{[0,1/9]} v^t|\log v|^j  \right) \begin{cases} 2N \qquad  & \textrm{if} ~ \kappa=0, \\ \frac{28}{9}N +3400  & \textrm{if} ~ \kappa=1. \end{cases}
\end{equation}
In the second sum over $\xi_n\notin X$ we have $G(\xi_n)\in [1/9,9]$, hence $|\log G(\xi_n)|\leq \log 9$, so bringing out this estimate from the sum and then extending the summation to all $n$
leads to
\begin{align}\label{eq:GtGprimeljonX}
\sum_{\xi_n\notin X} G^{t}(\xi_n)|G'(\xi_n)|^\kappa |\log G(\xi_n)|^{j} & \leq \log^j 9 \sum_{n=1-N}^N G^{t}(\xi_n)|G'(\xi_n)|^\kappa = \log^j 9 ~ S(t,0,\kappa).
\end{align}
Summing up, if $j\ne 0$ then the upper estimate of
\begin{align}\label{eq:Qobtains}
\sum_{n=1-N}^N G^{t}(\xi_n)& |G'(\xi_n)|^{\kappa}|\log G(\xi_n)|^{j} \notag \\
& < \begin{cases}\left( \max_{[0,1/9]} v^t|\log v|^j  \right) 2N + \log^j9 \cdot 2 Q_N(G,t,0) ~& {\rm if} ~ \kappa=0, \\
\left( \max_{[0,1/9]} v^t|\log v|^j  \right) \left\{\frac{28}{9}N +3400 \right\} + \log^j9 \cdot 2 Q^{*}_N(G,t,0) ~& {\rm if} ~ \kappa=1.
\end{cases}
\end{align}
follows for the generic term, and so taking into account the notations \eqref{eq:Qtjplain} and \eqref{eq:Qtjstar}, from \eqref{eq:Qobtains} and \eqref{eq:quadriple} the lemma follows.
\end{proof}


\section{Derivatives of the difference function $d(t)$ at the left endpoint}\label{sec:derivativesatleftend}

With the improved quadrature we now calculate the values of $d'(5), d''(5)$ and $d'''(5)$ first.
\begin{lemma}\label{l:dprime5benpos} We have $d'(5)>0$.
\end{lemma}
\begin{remark} Actually, $d'(5)=0.00287849...$ by numerical calculation, but formally we don't need an a priori knowledge of the value. Of course, putting together the argument we needed to take it into account, but the proof of the Lemma is deductive.
\end{remark}
\begin{proof} From \eqref{eq:HIVgeneralrafkos}, substituting $j=1$ and $t=5$ and denoting, as elsewhere $\ell:=|L|=|\log G|$
\begin{align}\label{eq:HIVj1t5rafkos}
|H^{IV}|&\leq G |G'| \{839,573,504 + 654,213,120 \ell \} +  G^{2}|G'| \{ 337,497,600 + 430,848,000 \ell\} \notag
\\& + G^{3}|G'| \{10,080,000 + 22,400,000 \ell \}+ G^{3}\{1,248,480,000 + 2,774,400,000 \ell \}
\\& \notag + G^{4} \left\{11,600,000 + 58,000,000 \ell \right\}.
\end{align}
This estimate is of the form of condition \eqref{eq:ficonditions}, suitable for the application of our improved quadrature in Lemma \ref{l:superquadrature}, which we invoke with $N:=500$ here. Therefore, we compute the expressions \eqref{eq:Qtjplain} and \eqref{eq:Qtjstar} with $N=500$ and with the occurring pairs of values of $t$ and $j=0,1$ as follows.

First of all, observe that according to \cite[Lemma 6]{Krenci2} for $j=1$ and $t=1,2,3,4$ we have $\max\limits_{[0,1/9]} v^t|\log v|^j=9^{-t}\log^j 9 $ (as the maximum place $v_0=\exp(-j/t)$ is larger, than $1/9$). For the computation of $\int_\TT G^t$ and $\Var (G^t)$ we refer to Corollaries \ref{cor:Parseval} and \ref{cor:VarG}. These lead to
\begin{align*}
Q^{*}_{500}(G,1,0) & \leq 123,987 + 6509 + 6585 =137,081, \\
Q^{*}_{500}(G,1,1) &\leq \frac{\log 9}{9} \left( \frac{7000}{9}+1700\right) + \log 9\cdot 137,081  \leq 301,803, \\
Q^{*}_{500}(G,2,0)&\leq 616,734.71 + 43,643.12 + 42,973.37 =703,352, \\
Q^{*}_{500}(G,2,1)&\leq \frac{\log 9}{9^2}\left( \frac{7000}{9}+1700\right) +  \log 9 \cdot 703,352 \le 1,545,490,  \\
Q^{*}_{500}(G,3,0)&\leq 3,632,988 + 325,636 + 318,808  = 4,277,432, \\
Q^{*}_{500}(G,3,1)&\leq \frac{\log 9}{9^3}\left( \frac{7000}{9}+1700\right) +  \log 9 \cdot 4,277,432  \leq 9,398,487, \\
Q_{500}(G,3,0)&\leq 46,500 + 1851 = 48,351, \\
Q_{500}(G,3,1)&\leq \frac{\log 9}{9^3} 500 + \log 9 \cdot 48,351  \leq 106,240, \\
Q_{500}(G,4,0)&\leq 319,500 + 14,532  =334,032, \\
Q_{500}(G,4,1)&\leq \frac{\log 9}{9^4} 500 + \log 9 \cdot 334,032 \le 733,944 . \\
\end{align*}
It remains to apply Lemma \ref{l:superquadrature} both for $H_{+}$ and $H_{-}$ with the coefficients $B_r, D_r$ read from \eqref{eq:HIVj1t5rafkos} and the corresponding $Q_{500}(G,t,j)$, $Q^{*}_{500}(G,t,j)$
estimated according to the above list. Executing the numerical computations leads to
\begin{equation}\label{eq:HIVj1t5rafkonumeric}
\left| \int_0^{1/2} H_{\pm} - \sum_{n=1}^{500} \left\{ H_{\pm}\left(\frac{2n-1}{2000}\right) \frac{1}{1000} + H_{\pm}'' \left(\frac{2n-1}{2000}\right) \frac{1}{192\cdot500^3} \right\} \right|
\leq 0.0009745... .
\end{equation}
Thus the quadrature approximation to integrals of $H_{\pm}$ lead to approximate values within the error $\delta:=0.001$. This error estimation is applied to both $H_+$ and $H_-$. The approximate value of $d'(5)=\int_0^{1/2} H_{-}-\int_0^{1/2} H_{+}$ from the quadrature is found to be $0.002878492... >0.002$, while the total error incurred is still bounded by $2\delta$. Therefore, $d'(5)> 0.002 - 2 \delta = 0$ and the assertion is proved.
\end{proof}

\begin{lemma}\label{l:d2ndder5benpos} We have $d''(5)>0$.
\end{lemma}
\begin{remark} By numerical calculation, $d''(5)\approx 0.033815603$.
\end{remark}
\begin{proof}
Now we want to use the improved quadrature again, hence we start with substituting $t=5$, $j=2$ into formula \eqref{eq:HIVgeneralrafkos} to derive
\begin{align}\label{eq:HIVj2t5rafkos}\notag
|H^{IV}|&\leq G |G'| \{774,152,192 + 1,679,147,008 \ell+ 654,213,120 \ell^2 \}
\\ & \notag +  G^{2}|G'| \{ 172,339,200 + 674,995,200 \ell + 430,848,000 \ell^2\}
\\& + G^{3}|G'| \{2,240,000 + 20,160,000 \ell + 22,400,000 \ell^2\}
\\& \notag + G^{3} \{277,440,000 + 2,496,960,000 \ell + 2,774,400,000 \ell^2\}
\\ & \notag + G^{4} \left\{23,200,000 \ell + 58,000,000 \ell^2 \right\}.
\end{align}
Now we may set the step number to $N=400$. The values of the occurring $Q_{400}(G,t,j)$ and $Q^{*}_{400}(G,t,j)$ can now be estimated as follows.
\begin{center}
\begin{tabular}{|c|c|c|c|c|}

$Q^{*}_{400}(G,1,0) \leq 112,282$ & $Q^{*}_{400}(G,1,1)\leq 247,274$ & $Q^{*}_{400}(G,1,2)\leq 543,316$\\
$Q^{*}_{400}(G,2,0) \leq 580,005$ & $Q^{*}_{400}(G,2,1)\leq 1,274,463$ & $Q^{*}_{400}(G,2,2)\leq 2,800,281$\\
$Q^{*}_{400}(G,3,0) \leq 3,550,835$ & $Q^{*}_{400}(G,3,1)\leq 7,801,987$ & $Q^{*}_{400}(G,3,2)\leq 17,142,718$\\
$Q_{400}(G,3,0) \leq 39,051$ & $Q_{400}(G,3,1)\leq 85,804$ & $Q_{400}(G,3,2)\leq 188,530$\\
$Q_{400}(G,4,0)$ \textrm{does not occur} & $Q_{400}(G,4,1) \leq 593,541$ & $Q_{400}(G,4,2)\leq 1,304,143$ \\
\end{tabular}
\end{center}
Applying Lemma \ref{l:superquadrature} for either $H_{+}$ or $H_{-}$ with the coefficients $B_r, D_r$ read from \eqref{eq:HIVj2t5rafkos} and the corresponding $Q_N(G,t,j)$, $Q^{*}_N(G,t,j)$ above, the numerical computations yield
\begin{equation}\label{eq:HIVj2t5rafkonumeric}
\left| \int_0^{1/2} H_{\pm} - \sum_{n=1}^{400} \left\{ H_{\pm}\left(\frac{2n-1}{4\cdot400}\right) \frac{1}{2\cdot400} + H_{\pm}'' \left(\frac{2n-1}{4\cdot400}\right) \frac{1}{192\cdot400^3} \right\} \right|
\leq 0.0071... =:\delta.
\end{equation}
This quadrature error estimation is applied for both $H_+$ and $H_-$, so the total error incurred is still bounded by $2\delta$, while the approximate value of $d''(5)=\int_0^{1/2} H_{-}-\int_0^{1/2} H_{+}$ from the quadrature is found to be $\approx 0.033815603$. Therefore, $d''(5)> 0.033815603 - 2 \delta >0$.
\end{proof}

\begin{lemma}\label{l:diff3ndder5benpoz} We have $d'''(5)>0$.
\end{lemma}
\begin{remark} By numerical calculation, $d'''(5)\approx 0.183547634...$.
\end{remark}
\begin{proof}
Now it suffices to apply the less refined estimates 
from \eqref{eq:HIVgeneralestimatewMmin} with $t=5, j=3$ to get
\begin{align}
|H^{IV}(x)|&\leq  959,512,576 \cdot v\{ 84 +426\ell +462\ell^2
+ 120 \ell^3\}
\notag \\& +  1,263,820,800\cdot v^2 \{ 6+72\ell
+141 \ell^2+60\ell^3 \}
+  11,600,000 \cdot v^4 \left\{3\ell^2+5\ell^3 \right\}
\\& +  335,840,000 \cdot v^3  \{6\ell+27\ell^2+20\ell^3 \} \notag .
\end{align}
In this estimation all the occurring functions of type $v^s \ell^m$ have maximum on $[0,9]$ at the right endpoint $v=9$ in view of \cite[Lemma 6]{Krenci2}. Therefore we can further estimate substituting $\ell=\log 9$ and  $v=9$. Thus we finally obtain $|H^{IV}(x)|\leq 2.82932\cdot10^{14}$.

To bring the error below $\delta=0.091$ we chose the step number $N$ large enough to have
$$
\frac{2.83 \cdot 10^{14}}{60\cdot2^{10} N^4}<\delta
\qquad
\textrm{i.e.}
\qquad
N\geq N_0:=  \sqrt[4]{\frac{2.83 \cdot 10^{14}}{60\cdot 2^{10}\cdot0.091}}\approx 475....
$$
Calculating the quadrature formula with $N=500$, we obtain the approximate value $d'''(5)=\int_0^{1/2} H_{-} - \int_0^{1/2} H_{+} \approx 0.18354763424...$, whence  $d'''(5)> 0.18354763424... -2 \cdot0.091> 0$.
\end{proof}

\section{Signs of derivatives of $d(t)$ and conclusion of the proof of Conjecture \ref{conj:con3}}\label{sec:ConjProof}

After examining the values of derivatives of $d$ at the left endpoint $t=5$, now we divide the interval $[5,6]$ to 3 parts. First we will prove in Lemma \ref{l:dIVaround5065} that $d^{IV}(t)>0$ in $[5, 5.13]$. In view of the above proven Lemmas \ref{l:dprime5benpos}, \ref{l:d2ndder5benpos} and \ref{l:diff3ndder5benpoz}, it follows, that in this interval also $d'''(t), d''(t), d'(t)>0$.

Next we will consider $d'$ in the interval $[5.13, 5.72]$. Lemmas \ref{l:dIVaround523} and \ref{l:dIVaround5525} will furnish $d'>0$ also in this domain. Consequently, $d$ is increasing all along $[5,5.72]$, and as $d(5)=0$, it will be positive in $(5,5.72]$. Finally we will show Lemma \ref{l:d2ndclose6}, giving that $d(t)$ is concave in the interval $[5.72, 6]$. As $d(5.72)>0$ and $d(6)=0$, this entails that the function remains positive on $[5.72, 6)$, too, whence $d>0$ on the whole of $(5,6)$.

Now we compute a good approximation of $d^{IV}(t)$ on the interval $[5,5.13]$ and using it show that $d^{IV}(t)$ stays positive in this interval.

In $[5, 5.13]$ the fourth derivative of $d(t)$ has the Taylor approximation
\begin{align}\label{eq:dVTaylor5065}
d^{IV}(t)&=\sum_{j=0}^n \frac{d^{(j+4)}(5.065)}{j!}\left(t-5.065\right)^j +R_{n}(d^{IV},5.065,t),\qquad {\rm where} \\\ \notag
& R_{n}(d^{IV},5.065,t):=\frac{d^{(n+5)}(\xi)}{(n+1)!}\left(t-5.065\right)^{n+1}.
\end{align}
Therefore using \eqref{eq:djnormbyHL1} we can write
\begin{align}\label{eq:Rd4totod}
|R_n(d^{IV},5.065,t)|& \leq \frac{\|H_{\xi,n+5,+}\|_{L^1[0,1/2]} + \|H_{\xi,n+5,-}\|_{L^1[0,1/2]}}{(n+1)!}\cdot 0.065^{n+1} \notag \\ &\leq  \frac{\frac12\|H_{\xi,n+5,+}\|_\infty + \frac12\|H_{\xi,n+5,-}\|_\infty }{(n+1)!}\cdot 0.065^{n+1}\\ & \leq \frac{\max_{|\xi-5.065|\leq 0.065} \|H_{\xi,n+5,+}\|_\infty + \max_{|\xi-5.065|\leq 0.065} \|H_{\xi,n+5,-}\|_\infty}{(n+1)!}\cdot 0.065^{n+1}.\notag
\end{align}
So once again we need to maximize \eqref{eq:Hdef}, that is functions of the type $v^{\xi} |\log v|^m $, on $[0,9]$. From \cite[Lemma 6]{Krenci2} we get, say for all $n\leq 30$
\begin{equation}\label{eq:maxmax5065}
\max_{5\leq \xi \leq 5.13} \|H_{\xi,n+5,\pm}(x)\|_{\infty} \leq  \max_{\xi \in [5,5.13]} \max_{v\in [0,9]} v^{\xi} |\log v|^{n+5}
= 9^{5.13} \log^{n+5} 9.
\end{equation}
Choosing $n=6$ yields $\|H_{\xi,n+5,\pm}(x)\|_{\infty} \leq 452,775,589$, and the Lagrange remainder term \eqref{eq:Rd4totod} of the Taylor formula \eqref{eq:dVTaylor5065} can be estimated as $|R_n(d^{IV},t)|\leq 0.0008808... <0.0009=:\de_{7}$.

Now we have to calculate the value of $d^{(j)}(t)$ -- that is, the two integrals in \eqref{eq:djdef} -- numerically for $k=5$, $t=5$ and $j=4,5,\dots,10$ to determine the Taylor coefficients in the above expansion. However, this cannot be done  precisely, due to the necessity of some numerical integration in the calculation of the two integrals in formula \eqref{eq:djdef}. We apply our numerical quadrature to derive at least a good approximation.

Denote $\overline{d}_j\approx d_{j+4}$ the numerical quadrature approximations. We set $\de:=0.187$ and want that $|d^{IV}(t) - P_6(t)| < \de$ for
\begin{equation}\label{eq:Pndef5065}
P_6(t):=\sum_{j=0}^n \frac{\overline{d}_j}{j!}\left(t-5.065\right)^j.
\end{equation}
In order to achieve this, we set the partial errors $\de_0,\dots,\de_{6}$ with $\sum_{j=0}^{7}\de_j <\de$, and ascertain that the termwise errors in approximating the Taylor polynomial $T_6(d^{IV})$ by $P_6$ satisfy analogously as in \cite[(37)]{Krenci2}
\begin{equation}\label{eq:djoverbarcriteria5065}
\left\|\frac{d^{(j+4)}(5.065)-\overline{d}_j}{j!}\left(t-5.065\right)^j\right\|_\infty =\frac{\left|d^{(j+4)}(5.065)-\overline{d}_j\right|}{j!}\cdot0.065^j< \delta_j\quad (j=0,\dots,6).
\end{equation}
That the termwise error \eqref{eq:djoverbarcriteria5065} would not exceed $\de_j$ will be guaranteed by $N_j$ step quadrature approximation of the two integrals in \eqref{eq:djdef} defining $d^{(j+4)}(5.065)$ with prescribed error $\eta_j$ each. Therefore, we set $\eta_j:=\de_j j!/(2\cdot0.065^j)$, and note that in order to have \eqref{eq:djoverbarcriteria5065} \begin{equation}\label{eq:Njchoice5065}
N_j > N_j^{\star}:=\sqrt[4]{\frac{\|H^{IV}_{5.065,j+4,\pm}\|_\infty}{60\cdot 2^{10} \eta_j}} = \sqrt[4]{\frac{\|H^{IV}_{5.065,j+4,\pm}\|_\infty2\cdot0.065^j}{60\cdot 2^{10} j!  \de_j}}
\end{equation}
suffices by the quadrature formula \eqref{eq:quadrature}. That is, we must estimate $\|H^{IV}_{5.065,j+4,\pm}\|_\infty$ for $j=0,\dots,6$ and thus find appropriate values of $N_j^{\star}$.

\bigskip\begin{lemma}\label{l:HIVnorm5065} For $j=0,\dots,6$ we have the numerical estimates of Table \ref{table:HIVnormandNj5065} for the values of $\|H^{IV}_{5.065,j,\pm}\|_\infty$. Setting $\de_j$ for $j=0,\dots,6$ as is given in the table, the quadrature of order $500:=N_j\geq N_j^{\star}$ with the listed values of $N_j^{\star}$ yield the approximate values $\overline{d}_j$ as listed in Table \ref{table:HIVnormandNj5065}, admitting the error estimates \eqref{eq:djoverbarcriteria5065} for $j=0,\dots,6$. Furthermore, we have for the Lagrange remainder term $\|R_{6}(d^{IV},t)\|_{\infty} <0.0009=:\de_{7}$ and thus with the approximate Taylor polynomial $P_{6}(t)$ defined in \eqref{eq:Pndef5065} the approximation $|d^{IV}(t)-P_{6}(t)|<\de:=0.187$ holds uniformly in $[5,5.13]$.
\begin{table}[h!]
\caption{Estimates for values of $\|H^{IV}_{5.065,j+4,\pm}\|_\infty$, $\de_j$, $N_j^{\star}$ and $\overline{d_j}$ for $j=0,\dots,6$.}
\label{table:HIVnormandNj5065}
\begin{center}
\begin{tabular}{|c|c|c|c|c|}
$j$ \qquad & $\|H^{IV}_{5.065,j+4,\pm}\|_\infty$ & $\de_j$ & $N_j^{\star}$ & $\overline{d_j}$\\
0 & $ 9.28687\cdot 10^{14}$ & 0.15 & 474 & 0.381737508\\
1 & $ 2.52880 \cdot 10^{15}$ & 0.03 & 460 & -2.087768122\\
2 & $ 6.81644 \cdot 10^{15}$ & 0.005 & 392 & -23.85760346\\
3 & $ 1.82039 \cdot 10^{16}$ & 0.0005 & 342 & -140.6261273\\
4 & $ 4.82014 \cdot 10^{16}$ & 0.0002 & 196 & -641.9545799\\
5 & $ 1.28469 \cdot 10^{17}$ & 0.0002 & 85 & -2521.387336\\
6 & $ 3.80117 \cdot 10^{17}$ & 0.0002 & 36 & -8940.14559\\

\end{tabular}
\end{center}
\end{table}
\end{lemma}\bigskip
\begin{proof} We start with the numerical upper estimation of $H^{IV}_{5.065,j,\pm}(x)$ for $x\in \TT$. For that, now we substitute $t=5.065$ in the general formula \eqref{eq:HIVgeneralestimatewMmin}. This results in
\begin{align}
|H_{5.065,j,\pm}^{IV}(x)|&\leq 959,512,576 \cdot v^{1.065}\Big\{ j(j-1)(j-2)(j-3)\ell^{j-4} + 14.26j(j-1)(j-2)\ell^{j-3}
\notag \\ & \qquad +73.75535j(j-1)\ell^{j-2} +163.4085485j\ell^{j-1}
+ 130.313837600625 \ell^j\Big\}
\notag \\& + 1,263,820,800\cdot  v^{2.065} \Big\{ j(j-1)(j-2)\ell^{j-3}+12.195j(j-1)\ell^{j-2}
+48.572675j \ell^{j-1}
\notag \\&\qquad +63.105974625\ell^j \Big\}
+ 11,600,000 \cdot v^{4.065} \left\{j\ell^{j-1}+5.065\ell^j \right\}
\\& + 335,840,000 \cdot v^{3.065} \Big\{j(j-1)\ell^{j-2}+8.13j\ell^{j-1}+20.589225\ell^j \Big\} \notag .
\end{align}
Otherwise, almost all the functions $v^s \ell^m$ (with $\ell:=|\log v|$) occurring here satisfy that their maximum on $0\leq v \leq 9$ is achieved at the right endpoint $v=9$. By \cite[Lemma 6]{Krenci2}, equation (14) this is the case whenever $m/s \leq 1/\sigma_0$; note that here we consider the degree 6 Taylor polynomial of $d^{IV}$, which entails $m \leq 10$ while the minimal occurring value of $s$ is $s=1.065$. So checking the condition $m/s \leq 1/\sigma_0 \approx 1/0.126 \approx 7.9365..$, we obtain that $v=9$ remains the actual maximum place except for $s=1.065$ and $m=9$ or $10$. It occurs two times: when $\max_{v\in [0,9]} v^{1.065} \ell^9 = (9/(e\cdot 1.065))^9= 27126.00128...$ when $j=9$, and when $j=10$ as power $j-1$; together with for $s=1.065$ and $m=10$, when $\max_{v\in [0,9]} v^{1.065} \ell^{10} = (10/(e\cdot 1.065))^{10}= 241857.246...$.

We collect the resulting numerical estimates of $\|H^{IV}\|$ in Table \ref{table:HIVnormandNj5065} and list the corresponding values of $N_j^{\star}$ read from formula \eqref{eq:Njchoice5065}. Moreover, we list in the table the values of $\overline{d_j}$, too, as furnished by the numerical quadrature formula \eqref{eq:quadrature} with step size $h=0.001$, i.e. $N=N_j=500> N_j ^{\star}$ ($j=0,\dots,6$) steps.
\end{proof}
\begin{lemma}\label{l:dIVaround5065} We have $d^{IV}(t)>0$ for all $5 \leq t \leq 5.13$.
\end{lemma}
\begin{proof} We approximate $d^{IV}(t)$ by the polynomial $P_{6}(t)$ constructed in \eqref{eq:Pndef5065} as the approximate value of the order 5 Taylor polynomial of $d^{IV}$ around $t_0:=5.065$. As the error is at most $\de=0.187$, it suffices to show that $p(t):=P_{5}(t)-\de>0$ in $[5,5.13]$. Now $P_{6}(5.13)=0.188694031...$ so $p(5)=P_{6}(5.13)-\de =0.188694031... -0.187 > 0$.

Moreover, $p'(t)=P_{6}'(t)=\sum_{j=1}^{6} \dfrac{\overline{d}_j}{(j-1)!} (t-5.065)^{j-1}$ and $p'(5)=-0.806502699...<0$. From the explicit formula of $p(t)$ we consecutively compute also $p''(5)=-15.96427771...<0$, $p'''(5)=-103.8163124...<0$, $p^{(4)}(5)=-496.9504606...<0$ and $p^{(5)}(5)=-1940.277873...<0$. Finally, we arrive at $p^{(6)}(t)=\overline{d}_6 =-8940.14559...$, which is constant, so $p^{(6)}(t)<0$ for all $t \in \RR$. From the found negative values at 5 it follows consecutively that also $p^{(5)}(t)<0$, $p^{(4)}(t)<0$, $p'''(t)<0$, $p''(t)<0$ and $p'(t)<0$ on $[5,5.13]$. Therefore, $p$ is decreasing, and as $p(5.13)>0$, $p(t)>0$ on the whole interval $5 \leq t\leq 5.13$.
\end{proof}

Next we set forth proving that $d'(t)>0$ for all $t\in [5.13,5.72]$. In this interval we use the refined process, applying \eqref{eq:superquadrature}. Still, for the entire interval we obtain step numbers $N\approx550$, so in order to push down $N$ under $500$, we divide the interval into 2 parts, and apply the method for both sections $[5.13, 5.33]$ and $[5.33, 5.72]$ separately. That is, we construct approximating Taylor polynomials around $5.23$ and $5.525$.

So now setting $t_0=5.23$ or $t_0=5.525$, the Taylor approximation of radii $r_0=0.1$ and $r_0=0.195$, respectively,  will have the form
\begin{equation}\label{eq:dVTaylor523}
d'(t)=\sum_{j=0}^n \frac{d^{(j+1)}(t_0)}{j!}\left(t-t_0\right)^j +R_{n}(d',t_0,t),\quad
R_{n}(d',t_0,t):=\frac{d^{(n+2)}(\xi)}{(n+1)!}\left(t-t_0\right)^{n+1}.
\end{equation}
Therefore instead of \cite[(36)]{Krenci2} we can use
\begin{align}\label{eq:Rd1totod}
|R_n(d',t_0,t)|& \leq \frac{\|H_{\xi,n+2,+}\|_{L^1[0,1/2]} + \|H_{\xi,n+2,-}\|_{L^1[0,1/2]}}{(n+1)!}\cdot r_0^{n+1} \notag \\ &\leq  \frac{\frac12\|H_{\xi,n+2,+}\|_\infty + \frac12\|H_{\xi,n+2,-}\|_\infty }{(n+1)!}\cdot r_0^{n+1}\\ & \leq \frac{\max_{|\xi-t_0|\leq r_0} \|H_{\xi,n+2,+}\|_\infty + \max_{|\xi-t_0|\leq r_0} \|H_{\xi,n+2,-}\|_\infty}{(n+1)!}\cdot r_0^{n+1}.\notag
\end{align}
So once again we need to maximize \eqref{eq:Hdef}, that is functions of the type $|\log v|^m v^{\xi}$, on $[0,9]$. From \cite[Lemma 6]{Krenci2} and as $\xi \geq 5$ for all cases, we obtain for all $n+2\leq 5/\sigma_0 \approx 39.68...$, i.e. for $n\leq 37$ \begin{equation}\label{eq:maxmax523}
\max_{|\xi-t_0|\leq r_0} \|H_{\xi,n+2,\pm}(x)\|_{\infty} \le \max_{|\xi-t_0|\leq r_0} \max_{0\leq v\leq 9} v^{\xi} |\log v|^{n+2} \leq  9^{t_0+r_0} \log^{n+2} 9.
\end{equation}
Consider the case $t_0=5.23$. We find (executing numerical tabulation of values  for orientation), that $d'$ is increasing from $d'(5.13)\approx 0.0089834...$ to even more positive values as $t$ increases from 5.13 to 5.33. This suggest that it will suffice to approximate $d'$ with an overall error just below $d'(5.13)\approx 0.0089834...$.

We now chose $n=8$, when according to \eqref{eq:maxmax523} $\|H_{\xi,10,\pm}(x)\|_{\infty} \leq 319,784,241$. Therefore the Lagrange remainder term \eqref{eq:Rd1totod} of the Taylor formula \eqref{eq:dVTaylor523} with $n=8$ can be estimated as
$|R_8(d',t)|\leq 0.00000176248\dots < 0.000002=:\de_{9}$.

As before, the Taylor coefficients $d^{(j+1)}(5.23)$ cannot be obtained exactly, but only with some error, due to the necessity of some kind of numerical integration in the computation of the formula \eqref{eq:djdef}. Hence we must set the partial errors $\de_0,\dots,\de_{8}$ in $|\frac{\overline{d_j}-d^{(j+1)}(5.23)}{j!} (t-5.23)^j|<\delta_j$ such that their sum would satisfy $\sum_{j=0}^{9}\de_j =:\de<0.0089834$ in order to have that at least $d'(5.13)> P_8(5.13) - \de >0$ for the approximate Taylor polynomial
\begin{equation}\label{eq:Pndef523}
P_8(t):=\sum_{j=0}^8 \frac{\overline{d}_j}{j!}\left(t-5.23\right)^j.
\end{equation}
In order to achieve this, we set the partial errors $\de_0,\dots,\de_{8}$ with $\sum_{j=0}^{9}\de_j <\de$, and ascertain that the termwise errors in approximating the Taylor polynomial $T_8(d')$ by $P_8$ satisfy analogously to \cite[(37)]{Krenci2}
\begin{equation}\label{eq:djoverbarcriteria523}
\left\|\frac{d^{(j+1)}(5.23)-\overline{d}_j}{j!}\left(t-5.23\right)^j\right\|_\infty =\frac{\left|d^{(j+1)}(5.23)-\overline{d}_j\right|}{j!}\cdot0.1^j< \delta_j\quad (j=0,\dots,8).
\end{equation}
We use the refined quadrature \eqref{eq:superquadrature} setting $N=500$ for all $j=1,...,8$. For this, first we need some estimate of the form \eqref{eq:ficonditions} for $|H^{IV}_{5.23,j+1,\pm}|$ and for all $j=0,\dots,8$. Once such an estimate is found with certain exponents $(t_r,j_r)$ and corresponding coefficients $B_r,D_r$, the improved quadrature formula \eqref{eq:superquadrature} furnishes an error estimate by means of
\begin{equation}\label{eq:Wdef}
W:=W({\bf B, D, t, j}):=\sum\limits_{r=1}^R \left\{ B_r Q_N(G,t_r,j_r) + D_r ~Q^{*}_N(G,t_r,j_r) \right\}.
\end{equation}
Namely, the error bound of numerical integration by using our quadrature will then be $\eta_j=\dfrac{W}{60\cdot2^{10}\cdot500^5}$, and the corresponding termwise error bound becomes $\de_j=\dfrac{\eta_j}{2(j-1)!\cdot 0.1^{j-1}}$.
\begin{lemma}\label{l:HIVnorm523} For $j=0,\dots, 8$ we have the numerical estimates of Table \ref{table:HIVnormandNj523} for the values of $W$. Setting $\de_j$ as given in the table for $j=0,\dots, 8$, the approximate quadrature of order $N_j:=N:=500$ yield the approximate values $\overline{d}_j$ as listed in Table \ref{table:HIVnormandNj523}, admitting the error estimates \eqref{eq:djoverbarcriteria523} for $j=0,\dots, 8$. Furthermore, $\|R_{9}(d',t)\|_{\infty} <0.000002=:\de_{9}$ and thus with the approximate Taylor polynomial $P_{8}(t)$ defined in \eqref{eq:Pndef523} the approximation $|d'(t)-P_{8}(t)|<\de:=0.004784113$ holds uniformly for $ t \in [5.13,5.33]$.
\end{lemma}
\begin{table}[h!]
\caption{Estimates for values of $W$ and $\de_j$, $\overline{d_j}$ for $j=0,\dots,8$, with $N=500$.}
\label{table:HIVnormandNj523}
\begin{center}
\begin{tabular}{|c|c|c|c|c|}
$j$ \qquad & estimate for $W$ & $\de_j$  & $\overline{d_j}$\\
0 & $ 3.46227\cdot 10^{15}$ & 0.003606534 & 0.016265345\\
1 & $ 9.78474 \cdot 10^{15}$ & 0.001019244 & 0.084372338\\
2 & $ 2.73203 \cdot 10^{16}$ & 0.000142293 & 0.223408446\\
3 & $ 7.54351 \cdot 10^{16}$ & $1.30964 \cdot 10^{-5}$ & -0.41545758\\
4 & $ 2.06152 \cdot 10^{17}$ & $8.94756 \cdot 10^{-7}$ & -8.507038066\\
5 & $ 5.5806 \cdot 10^{17}$ & $4.84427 \cdot 10^{-8}$ & -57.99608037\\
6 & $ 1.4977 \cdot 10^{18}$ & $2.16681 \cdot 10^{-9}$ & -288.5739971\\
7 & $ 3.98926 \cdot 10^{18}$ & $8.24499 \cdot 10^{-11}$ & -1204.823065\\
8 & $ 1.05675 \cdot 10^{19}$ & $2.7301 \cdot 10^{-12}$ & -4474.521416\\

\end{tabular}
\end{center}
\end{table}
\begin{proof}
Substituting $t=5.23$ in \eqref{eq:HIVgeneralrafkos} yields
\comment{
\begin{align}
|H^{IV}(x)|&\leq G^{1.23}|G'|M_1^3\Big\{ j(j-1)(j-2)(j-3)|\log L|^{j-4} + 14.92j(j-1)(j-2)|\log L|^{j-3}
\notag \\& +80.9774j(j-1)|\log L|^{j-2} +94.465234j|\log L|^{j-1}
+ 159.34903641 |\log L|^j\Big\}
\notag \\& + 6\cdot G^{2.23}|G'|M_1 M_2 \Big\{ j(j-1)(j-2)|\log L|^{j-3}+12.69j(j-1)|\log L|^{j-2}
\notag \\& +52.6787j |\log L|^{j-1}+71.456967|\log L|^j \Big\}
+ G^{3.23}|G'|4M_3 \Big\{j(j-1)|\log L|^{j-2}
\notag \\&+9.46j|\log L|^{j-1}+22.1229|\log L|^j \Big\}+ G^{3.23}3M_2^2 \Big\{j(j-1)|\log L|^{j-2}
\notag \\&+9.46j|\log L|^{j-1}+22.1229|\log L|^j \Big\}+ G^{4.23} M_4 \left\{j|\log L|^{j-1}+5.23|\log L|^j \right\} \notag .
\end{align}
}
\begin{align}
|H^{IV}(x)|&\leq G^{1.23}|G'|5451776\Big\{ j(j-1)(j-2)(j-3)|\log L|^{j-4} + 14.92j(j-1)(j-2)|\log L|^{j-3}
\notag \\& +80.9774j(j-1)|\log L|^{j-2} +94.465234j|\log L|^{j-1}
+ 159.34903641 |\log L|^j\Big\}
\notag \\& + G^{2.23}|G'|7180800\Big\{ j(j-1)(j-2)|\log L|^{j-3}+12.69j(j-1)|\log L|^{j-2}
\notag \\& +52.6787j |\log L|^{j-1}+71.456967|\log L|^j \Big\}
+ G^{3.23}|G'|1120000 \Big\{j(j-1)|\log L|^{j-2}
\notag \\&+9.46j|\log L|^{j-1}+22.1229|\log L|^j \Big\}+ G^{3.23}138720000 \Big\{j(j-1)|\log L|^{j-2}
\notag \\&+9.46j|\log L|^{j-1}+22.1229|\log L|^j \Big\}+ G^{4.23} 11600000 \left\{j|\log L|^{j-1}+5.23|\log L|^j \right\} \notag .
\end{align}
Considering sums of $|H^{IV}(\xi_n)|$, this will be estimated by means of Lemma \ref{l:superquadrature}. So we insert values of $j, t$ and apply Lemma \ref{l:superquadrature} with step number $N=500$, getting estimations for $W$ as is shown in Table \ref{table:HIVnormandNj523}. We also calculate $\de_j=\dfrac{W/(60 \cdot 2^{10} \cdot 500^5)}{2(j-1)!\cdot 0.1^{j-1}}$.
\end{proof}
\begin{lemma}\label{l:dIVaround523} We have $d'(t)>0$ for all $5.13 \leq t \leq 5.33$.
\end{lemma}
\begin{proof} We approximate $d'(t)$ by the polynomial $P_{8}(t)$ constructed in \eqref{eq:Pndef523} as the approximate value of the order 8 Taylor polynomial of $d'$ around $t_0:=5.23$. As the error of this approximation is at most $\de$, it suffices to show that $p(t):=P_{8}(t)-\de>0$ in $[5.13,5.33]$. Moreover, $p'(t)=P_{8}'(t)=\sum_{j=1}^{8} \dfrac{\overline{d}_j}{(j-1)!} (t-5.23)^{j-1}$. Now $P_{8}(5.13)=0.008983405...$, and $P_{8}(5.33)=0.025709673...$,  so $P_{8}(5.13)-\de>0$ and $P_{8}(5.33)-\de>0$. If we suppose, that $p$ attain $0$ in this interval, that means, the total variation here $\Var(p)\ge P_{8}(5.13)+P_{8}(5.33)-2\de$. As $\Var(p)=\int_{5.13}^{5.33} |p'|dt$, we have an estimation for the integral mean of $|p'|$ note $I_{p'}\le \Var(p)/0.2 =0.12546539...$. As it is greater then $\max(|p'(5.13)|, |p'(5.33)|)$ and the continuous function has to attain its integral mean, we have an estimation for total variation of $p'$: $\Var(p')\le 2I_{p'}-|p'(5.13)+p'(5.33)|$. We also have an estimation for integral mean of $|p''|$: $I_{p''}\le \Var(p')/0.2 =0.43413663...$. This process can be continued, till $p^{(5)}$ (see Table \ref{table:Taylor523}). On the other hand, from the explicit formula of $p(t)$ we consecutively compute also $p^{(5)}(5.13)<0$, $p^{(6)}(5.13)<0$... Finally, we arrive at $p^{(8)}(t)=\overline{d}_8$=-4474.521416... However, $p^{(8)}$ is constant, so $p^{(7)}(t)<0$ and $p^{(8)}(t)<0$ in $[5.13, 5.33]$. It means, that $p^{(4)}(t)$ is decreasing in the interval. It is contradiction, as the calculated lower bound for integral mean of $|p^{(4)}|$ is greater than $\max(|p^{(4)}(5.13)|, |p^{(4)}(5.13)|)$, and the function should attain this value.
\end{proof}
\begin{table}[h!]
\caption{Estimates for values of $p^{(j)}(5.13)$, $p^{(j)}(5.33)$ and total variation, integral mean of $p^{(j)}$ on interval $[5.13, 5.33]$ for $j=0,\dots,8$}.
\label{table:Taylor523}
\begin{center}
\begin{tabular}{|c|c|c|c|c|}
$j$ \qquad & $p^{(j)}(5.13)$ & $p^{(j)}(5.33)$  & $\Var(p^{(j)})$ & integral mean $I_{p^{(j)}}$\\
0 & $0.0089834050$ & $0.025709673$ & $0.0250930779$ & \\
1 & $0.061152858$ & $0.102950595$ & $0.086827326$ & $0.12546539$\\
2 & $0.230976823$ & $0.128352476$ & $0.508943962$ & $0.43413663$\\
3 & $0.188714272$ & $-1.609630427$ & $3.66852346$ & $2.544719808$\\
4 & $-3.968140009$ & $-15.96896377$ & $16.74813082$ & $18.3426173$\\
5 & $-34.41704242$ & $-93.62334897$ & & \\
6 & $-190.4642977$ & $-431.4289106$ & & \\
7 & $-757.3709229$ & $-1652.275206$ & & \\
8 & $-4474.521416$ & $-4474.521416$ & & \\

\end{tabular}
\end{center}\bigskip
\end{table}

In case $t_0=5.525$ numerical tabulation of values gives that $d'$ is positive as $t$ increases from 5.33 to 5.72, and $d'(5.33)\approx 0.025709673\dots$, $d'(5.72)\approx 0.034577102$. We chose $n=9$ then $\|H_{\xi,n+2,\pm}(x)\|_{\infty} \leq 1,655,335,712$, for this case the Lagrange remainder term \eqref{eq:Rd1totod} of the Taylor formula \eqref{eq:dVTaylor523} can be estimated as
$|R_n(d',t)|\leq 0.0000725269\dots \leq 0.000073=:\de_{9}$.
Similarly to \eqref{eq:Pndef523} and \eqref{eq:djoverbarcriteria523} we now write
\begin{equation}\label{eq:Pndef5525}
P_n(t):=\sum_{j=0}^n \frac{\overline{d}_j}{j!}\left(t-5.525\right)^j,
\end{equation}
\begin{equation}\label{eq:djoverbarcriteria5525}
\left\|\frac{d^{(j+1)}(5.525)-\overline{d}_j}{j!}\left(t-5.525\right)^j\right\|_\infty =\frac{\left|d^{(j+1)}(5.525)-\overline{d}_j\right|}{j!}\cdot0.195^j< \delta_j\qquad (j=0,1,\dots,n).
\end{equation}
We also use the refined quadrature \eqref{eq:superquadrature} setting $N=500$ for all $j=1,...,9$. For this, first we need some estimate of the form \eqref{eq:ficonditions} for $|H^{IV}_{5.525,j+1,\pm}|$ and for all $j=0,\dots,9$. As before, once such an estimate is found with certain exponents $(t_r,j_r)$ and corresponding coefficients $B_r,D_r$, the improved quadrature formula \eqref{eq:superquadrature} furnishes an error estimate by means of $W$ defined in \eqref{eq:Wdef}, with the error of the quadrature being $\eta_j=\dfrac{W}{60\cdot2^{10}\cdot500^5}$, and the error of the corresponding term arising from the quadrature becoming $\de_j=\dfrac{\eta_j}{2(j-1)!\cdot 0.195^{j-1}}$.
\begin{lemma}\label{l:HIVnorm5525} For $j=0,\dots, 9$ we have the numerical estimates of Table \ref{table:HIVnormandNj5525} for W. Setting $\de_j$ as given in the table for $j=0,\dots, 9$, the approximate quadratures of order $N:=500$ yield the approximate values $\overline{d}_j$ as listed in Table \ref{table:HIVnormandNj5525}, admitting the error estimates \eqref{eq:djoverbarcriteria5525} for $j=0,\dots, 9$. Furthermore, $\|R_{10}(d^{IV},t)\|_{\infty} <0.000073=:\de_{10}$ and thus with the approximate Taylor polynomial $P_{9}(t)$ defined in \eqref{eq:Pndef5525} the approximation $|d'(t)-P_{9}(t)|<\de:=0.0124555$ holds uniformly for $ t \in [5.33,5.72]$.
\end{lemma}
\begin{table}[h!]
\caption{Estimates for values of $W$ and $\de_j$, $\overline{d_j}$ for $j=0,\dots,9$, with $N=500$.}
\label{table:HIVnormandNj5525}
\begin{center}
\begin{tabular}{|c|c|c|c|c|}
$j$ \qquad & estimate for $W$ & $\de_j$  & $\overline{d_j}$\\
0 & $ 6.89883\cdot 10^{15}$ & $0.007186277$ & $0.045016622$\\
1 & $ 1.93082 \cdot 10^{16}$ & $0.003921976$ & $0.070827581$\\
2 & $ 5.34273 \cdot 10^{16}$ & $0.00105811$ & $-0.6357179$\\
3 & $ 1.46288 \cdot 10^{17}$ & $0.000188317$ & $-7.162905157$\\
4 & $ 3.96656 \cdot 10^{17}$ & $2.48926 \cdot 10^{-5}$ & $-45.0748687$\\
5 & $ 1.06584 \cdot 10^{18}$ & $2.60863 \cdot 10^{-6}$ & $-220.5767067$\\
6 & $ 2.84009 \cdot 10^{18}$ & $2.2591 \cdot 10^{-7}$ & $-922.6394344$\\
7 & $ 7.50943 \cdot 10^{18}$ & $1.66398 \cdot 10^{-8}$ & $-3454.236354$\\
8 & $ 1.97155 \cdot 10^{19}$ & $1.06486 \cdot 10^{-9}$ & $-11,901.56441$\\
9 & $ 5.14412 \cdot 10^{19}$ & $6.01989 \cdot 10^{-11}$ & $-38,448.6079$\\
\end{tabular}
\end{center}
\end{table}
\bigskip
\begin{proof}
Substituting $t=5.525$ in \eqref{eq:HIVgeneralrafkos} yields\
\comment{
\begin{align}
|H^{IV}(x)|&\leq G^{1.525}|G'|M_1^3\Big\{ j(j-1)(j-2)(j-3)|\log L|^{j-4} + 16.1j(j-1)(j-2)|\log L|^{j-3}
\notag \\& + 94.70375j(j-1)|\log L|^{j-2} + 120.35253125 j|\log L|^{j-1}
+ 222.521187890625 |\log L|^j\Big\}
\notag \\& + 6\cdot G^{2.525}|G'|M_1 M_2 \Big\{ j(j-1)(j-2)|\log L|^{j-3}+13.575 j(j-1)|\log L|^{j-2}
\notag \\& + 60.426875 j |\log L|^{j-1}+ 88.127203125 |\log L|^j \Big\}
+ G^{3.525}|G'|4M_3 \Big\{j(j-1)|\log L|^{j-2}
\notag \\&+ 10.05 j|\log L|^{j-1}+ 25.000625|\log L|^j \Big\}+ G^{3.525}3M_2^2 \Big\{j(j-1)|\log L|^{j-2}
\notag \\&+ 10.05j|\log L|^{j-1}+ 25.000625|\log L|^j \Big\}+ G^{4.525} M_4 \left\{j|\log L|^{j-1}+5.525|\log L|^j \right\} \notag .
\end{align}
}
\begin{align}
|H^{IV}(x)|&\leq G^{1.525}|G'|5,451,776\Big\{ j(j-1)(j-2)(j-3)|\log L|^{j-4} + 16.1j(j-1)(j-2)|\log L|^{j-3}
\notag \\& + 94.70375j(j-1)|\log L|^{j-2} + 120.35253125 j|\log L|^{j-1}
+ 222.521187890625 |\log L|^j\Big\}
\notag \\& +  G^{2.525}|G'|7,180,800 \Big\{ j(j-1)(j-2)|\log L|^{j-3}+13.575 j(j-1)|\log L|^{j-2}
\notag \\& + 60.426875 j |\log L|^{j-1}+ 88.127203125 |\log L|^j \Big\}
+ G^{3.525}|G'|1,120,000 \Big\{j(j-1)|\log L|^{j-2}
\notag \\&+ 10.05 j|\log L|^{j-1}+ 25.000625|\log L|^j \Big\}+ G^{3.525}138,720,000 \Big\{j(j-1)|\log L|^{j-2}
\notag \\&+ 10.05j|\log L|^{j-1}+ 25.000625|\log L|^j \Big\}+ G^{4.525} 11,600,000 \left\{j|\log L|^{j-1}+5.525|\log L|^j \right\} \notag .
\end{align}
 Now the quadrature formula error is to be 
 estimated by means of Lemma \ref{l:superquadrature}. So we insert the values of $j, t$ and apply Lemma \ref{l:superquadrature} with $N=500$, obtaining the estimations for $W$ as is shown in Table \ref{table:HIVnormandNj5525}. We also get a value of $\de_j$ calculating $\de_j=\dfrac{W/(60\cdot2^{10} \cdot 500^5)}{2(j-1)!\cdot 0.195^{j-1}}$.
\end{proof}
\begin{lemma}\label{l:dIVaround5525} We have $d'(t)>0$ for all $5.33 \leq t \leq 5.72$.
\end{lemma}
\begin{proof} We approximate $d'(t)$ by the polynomial $P_{9}(t)$ constructed in \eqref{eq:Pndef5525} as the approximate value of the order 9 Taylor polynomial of $d'$ around $t_0:=5.525$. As the error is at most $\de$, it suffices to show that $p(t):=P_{9}(t)-\de>0$ in $[5.33,5.72]$. To apply the same method as in Lemma \ref{l:dIVaround523}, we divide the interval into two parts: $[5.33, 5,56]$ and $[5.56, 5.72]$. Moreover, $p'(t)=P_{9}'(t)=\sum_{j=1}^{9} \dfrac{\overline{d}_j}{(j-1)!} (t-5.525)^{j-1}$. Now $P_{9}(5.33)=0.025709673...$, $P_{9}(5.56)=0.047052108...$, and $P_{9}(5.72)=0.034577105...$  so in these points $P_{9}-\de>0$. In the interval [5.33, 5.56] we get for the integral mean of $|p^{(j)}|$ (estimating with total variation, as in Lemma \ref{l:dIVaround523}), $I_{p^{(j)}}>\max(|p^{(j)}(5.33)|, |p^{(j)}(5.56)|)$ for $j=1,2$ (see Table \ref{table:Taylor5525}). The function has to attain this estimated integral mean value, so it cannot be monotonic. On the other hand $p^{(j)}(5.33)<0$ for $j=3,\dots,9$. But $p^{(9)}$ is a constant, that is negative in the entire interval, hence $p^{(j)}$ are also negative in the whole interval for $j=3,\dots,8$. It follows that $p''$ is decreasing in the interval, which is a contradiction.

In case of the interval [5.56, 5.72] the process is similar: $I_{p'}>\max(|p'(5.56)|, |p'(5.72)|)$, and $p^{(j)}(5.56)<0$ for $j=2,\dots,9$, while $p^{(9)}$ is a constant, so $p^{(j)}$ are also negative in the interval for $j=2,\dots,9$. So $p'$ should be monotonic in the interval, and it is a contradiction.
\end{proof}

\begin{table}[h!]
\caption{Estimates for values of $p^{(j)}(5.33)$, $p^{(j)}(5.56)$, $p^{(j)}(5.72)$ and total variation, integral mean of $p^{(j)}$ on $[5.33, 5.56]$ and $[5.56, 5.72]$ for $j=0,\dots,9$}.
\label{table:Taylor5525}
\begin{center}
\begin{tabular}{|p{5pt}|p{60pt}|p{60pt}|p{60pt}|p{60pt}|p{60pt}|p{60pt}|p{60pt}|}
$j$ \qquad & $p^{(j)}(5.33)$ & $p^{(j)}(5.56)$ & $p^{(j)}(5.72)$  & \small{$\Var(p^{(j)})$ in  [5.33, 5.56]} & \small{mean $I_{p^{(j)}}$ in [5.33, 5.56]} &
\small{$\Var(p^{(j)})$ in [5.56, 5.72]} & \small{mean $I_{p^{(j)}}$ in [5.56, 5.72]}
\\ \hline
0 & \small{$0.0257096753$} & \small{$0.047052108$} & \small{$0.034577105$} & \small{$0.0478507836$} &  & \small{$0.0567182135$} &\\
1 & \small{$0.102950466$} & \small{$0.043853873$} & \small{$-0.260773968$} & \small{$0.269289432$} & \small{$0.208046885$} & & \small{$0.354488834$} \\
2 & \small{$0.12835791$} & \small{$-0.915663374$} & \small{$-3.226753649$} & & \small{$1.170823618$} & & \\
3 & \small{$-1.609886707$} & \small{$-8.882443109$} & \small{$-21.52543175$} & & & & \\
4 & \small{$-15.96198625$} & \small{$-53.38561447$} & \small{$-110.7051556$} & & & & \\
5 & \small{$-93.94395303$} & \small{$-255.0722573$} & \small{$-483.1895366$} & & & & \\
6 & \small{$-427.8265683$} & \small{$-1051.102162$} & \small{$-1870.009287$} & & & & \\
7 & \small{$-1864.435452$} & \small{$-3894.340881$} & \small{$-6506.045571$} & & & & \\
8 & \small{$-4404.08587$} & \small{$-13,247.2656$} & \small{$-19,399.0429$} & & & & \\
9 & \small{$-38,448.6078$} & \small{$-38,448.6078$} & \small{$-38,448.6078$} & & & & \\

\end{tabular}
\end{center}\bigskip
\end{table}\bigskip

The last step is to prove that $d$ is concave in the interval $[5.72, 6]$.
\begin{lemma}\label{l:d2ndclose6} We have $d''(t)<0$ for $5.72\leq t\leq 6$.
\end{lemma}
Numerical tabulation of values give that $d''$ is decreasing from $d''(5.72)\approx -0.260774...$ to even more negative values as $t$ increases from 5.72 to 6. In the interval $[5.72, 6]$ the second derivative of $d(t)$ has the Taylor-approximation
\begin{align}\label{eq:dVTaylor586}
d''(t)&=\sum_{j=0}^n \frac{d^{(j+2)}(5.86)}{j!}\left(t-5.86\right)^j +R_{n}(d'',5.86,t),\qquad {\rm where} \\\ \notag
& R_{n}(d'',5.86,t):=\frac{d^{(n+3)}(\xi)}{(n+1)!}\left(t-5.86\right)^{n+1}.
\end{align}
Therefore instead of \cite[(36)]{Krenci2} we can use
\begin{align}\label{eq:Rd2totod}
|R_n(d'',5.86,t)|& \leq \frac{\|H_{\xi,n+3,+}\|_{L^1[0,1/2]} + \|H_{\xi,n+3,-}\|_{L^1[0,1/2]}}{(n+1)!}\cdot 0.14^{n+1} \notag \\ &\leq  \frac{\frac12\|H_{\xi,n+3,+}\|_\infty + \frac12\|H_{\xi,n+3,-}\|_\infty }{(n+1)!}\cdot 0.14^{n+1}\\ & \leq \frac{\max_{|\xi-5.86|\leq 0.14} \|H_{\xi,n+3,+}\|_\infty + \max_{|\xi-5.86|\leq 0.14} \|H_{\xi,n+3,-}\|_\infty}{(n+1)!}\cdot 0.14^{n+1}.\notag
\end{align}
So once again we need to maximize \eqref{eq:Hdef}, that is functions of the type $|\log v|^m v^{\xi}$, on $[0,9]$. From \cite[Lemma 6]{Krenci2} it follows
\begin{equation}\label{eq:maxmax586}
\max_{|\xi-5.86|\leq 0.14} \|H_{\xi,n+3,\pm}(x)\|_{\infty} \leq  9^{6} \log^{n+3} 9.
\end{equation}
We chose $n=8$ then $\|H_{\xi,n+3,\pm}(x)\|_{\infty} \leq 3,062,485,120$, for this case the Lagrange remainder term \eqref{eq:Rd2totod} of the Taylor formula \eqref{eq:dVTaylor586} can be estimated as
$|R_n(d'',t)|\leq 0.011209281\dots <0.00035=:\de_{9}$.

As before, the Taylor coefficients $d_{j+2}(5.86)$ cannot be obtained exactly, but only with some error, due to the necessity of some kind of numerical integration in the computation of the formula \eqref{eq:djdef}. Hence we must set the partial errors $\de_0,\dots,\de_{8}$ with $\sum_{j=0}^{9}\de_j <\de:=0.2494$, say, so that $d''(t)< P_n(t) +\de$ for
\begin{equation}\label{eq:Pndef586}
P_n(t):=\sum_{j=0}^n \frac{\overline{d}_j}{j!}\left(t-5.86\right)^j.
\end{equation}
The analogous criteria to \cite[(37)]{Krenci2} now has the form:
\begin{equation}\label{eq:djoverbarcriteria586}
\left\|\frac{d^{(j+2)}(5.86)-\overline{d}_j}{j!}\left(t-5.86\right)^j\right\|_\infty =\frac{\left|d^{(j+2)}(5.86)-\overline{d}_j\right|}{j!}\cdot0.14^j< \delta_j\qquad (j=0,1,\dots,n).
\end{equation}
That the termwise error \eqref{eq:djoverbarcriteria586} would not exceed $\de_j$ will be guaranteed by $N_j$ step quadrature approximation of the two integrals in \eqref{eq:djdef} defining $d^{(j+2)}(5.86)$ with prescribed error $\eta_j$ each. Therefore, we set $\eta_j:=\de_j j!/(2\cdot0.14^j)$, and note that in order to have \eqref{eq:djoverbarcriteria586} \begin{equation}\label{eq:Njchoice586}
N_j > N_j^{\star}:=\sqrt[4]{\frac{\|H^{IV}_{5.86,j+2,\pm}\|_\infty}{60\cdot 2^{10} \eta_j}} = \sqrt[4]{\frac{\|H^{IV}_{5.86,j+2,\pm}\|_\infty2\cdot0.14^j}{60\cdot 2^{10} j!  \de_j}}
\end{equation}
suffices by the integral formula \eqref{eq:quadrature} and \cite[Lemma 5]{Krenci2}. That is, we must estimate $\|H^{IV}_{5.86,j+2,\pm}\|_\infty$ for $j=0,\dots,8$ and thus find appropriate values of $N_j^{\star}$.

\bigskip\begin{lemma}\label{l:HIVnorm586} For $j=0,\dots,8$ we have the numerical estimates of Table \ref{table:HIVnormandNj586} for the values of $\|H^{IV}_{5.86,j,\pm}\|_\infty$. Setting $\de_j$ as seeing in the table for $j=0,\dots,8$ and $\de_{9}=0.00035$, the approximate quadrature of order $N_j\geq N_j^{\star}$ with the listed values of $N_j^{\star}$ yield the approximate values $\overline{d}_j$ as listed in Table \ref{table:HIVnormandNj586}, admitting the error estimates \eqref{eq:djoverbarcriteria586} for $j=0,\dots,9$. Furthermore, $\|R_{9}(d'',t)\|_{\infty} <0.01121=:\de_{9}$ and thus with the approximate Taylor polynomial $P_{8}(t)$ defined in \eqref{eq:Pndef586} the approximation $|d''(t)-P_{8}(t)|<\de:=0.2494$ holds uniformly for $ t \in [5.72,6]$.
\begin{table}[h!]
\caption{Estimates for values of $\|H^{IV}_{5.86,j+2,\pm}\|_\infty$ and $\de_j$, $N_j^{\star}$ and $\overline{d_j}$ for $j=0,\dots,8$.}
\label{table:HIVnormandNj586}
\begin{center}
\begin{tabular}{|c|c|c|c|c|}
$j$ \qquad & $\|H^{IV}_{5.86,j+2,\pm}\|_\infty$ & $\de_j$ & $N_j^{\star}$ & $\overline{d_j}$\\
0 & $ 1.07968\cdot 10^{15}$ & 0.16 & 485 & -0.982761617\\
1 & $ 2.93801 \cdot 10^{15}$ & 0.062 & 483 & -7.57978318\\
2 & $ 7.91604 \cdot 10^{15}$ & 0.015 & 453 & -42.74047825\\
3 & $ 2.1135 \cdot 10^{16}$ & 0.002 & 446 & -200.2495965\\
4 & $ 5.59555 \cdot 10^{16}$ & 0.002 & 246 & -823.1734963\\
5 & $ 1.46998 \cdot 10^{17}$ & 0.002 & 128 & -3064.925687\\
6 & $ 3.83405 \cdot 10^{17}$ & 0.002 & 64 & -10,561.40925\\
7 & $ 9.93361 \cdot 10^{17}$ & 0.002 & 31 & -34,212.60072\\
8 & $ 2.55779 \cdot 10^{17}$ & 0.002 & 14 & -105,414.5993\\

\end{tabular}
\end{center}
\end{table}
\end{lemma}\bigskip
\begin{proof} We start with the numerical upper estimation of $H^{IV}_{5.86,j,\pm}(x)$ for $5.72\leq x\leq 6$. In the general formula \eqref{eq:HIVgeneralestimatewMmin} now we consider the case $t=5.86$.
\begin{align}
|H^{IV}(x)|&\leq  959,512,576 \cdot v^{1.86}\Big\{ j(j-1)(j-2)(j-3)\ell^{j-4} + 117.44j(j-1)(j-2)\ell^{j-3}
\notag \\ & \qquad +111.5576j(j-1)\ell^{j-2} +309.72742j\ell^{j-1}
+ 314.40339216 \ell^j\Big\}
\notag \\& +  1,263,820,800 \cdot v^{2.86} \Big\{ j(j-1)(j-2)\ell^{j-3}+14.58j(j-1)\ell^{j-2}
+69.8588j \ell^{j-1}
\notag \\&\qquad +109.931256\ell^j \Big\}
+  11,600,000 \cdot v^{4.86}\left\{j\ell^{j-1}+5.86\ell^j \right\}
\\& +  335,840,000 \cdot v^{3.86} \Big\{j(j-1)\ell^{j-2}+10.72j\ell^{j-1}+28.4796\ell^j \Big\} \notag .
\end{align}
Applying that all the occurring functions of type $v^s \ell^m$ have maximum on $[0,9]$ at the right endpoint $v=9$ in view of \cite[Lemma 6]{Krenci2}, we can further estimate
substituting $\ell=\log 9$ and  $v=9$. We collect the resulting numerical estimates of $\|H^{IV}\|_\infty$ in Table \ref{table:HIVnormandNj586} and list the corresponding values of $N_j^{\star}$ and $\overline{d_j}$, too, as given by the formulas \eqref{eq:Njchoice586} and the numerical quadrature formula \eqref{eq:quadrature} with step size $h=0.001$, i.e. $N=N_j=500$ steps.\end{proof}
\begin{lemma}\label{l:dIVaround586} We have $d''(t)<0$ for all $5.72 \leq t \leq 6$.
\end{lemma}
\begin{proof} We approximate $d''(t)$ by the polynomial $P_{8}(t)$ constructed in \eqref{eq:Pndef586} as the approximate value of the order 8 Taylor polynomial of $d''$ around $t_0:=5.86$. As the error is at most $\de$, it suffices to show that $p(t):=P_{8}(t)+\de<0$ in $[5.72,6]$. Now $P_{8}(5.72)=-0.2607741259...$ so $P_{8}(5.72)+\de<0$. Moreover, $p'(t)=P_{8}'(t)=\sum_{j=1}^{8} \dfrac{\overline{d}_j}{(j-1)!} (t-5.86)^{j-1}$ and $p'(5.72)=-3.226759...<0$. From the explicit formula of $p(t)$ we consecutively compute also $p''(5.72)=-21.525764...<0$, $p'''(5.72)=-110.671188...<0$, $p^{(4)}(5.72)=-483.626484...<0$, $p^{(5)}(5.72)=-1873.40227...<0$, $p^{(6)}(5.72)=-6804.70822...<0$, $p^{(7)}(5.72)=-19,454.5568...<0$. Finally, we arrive at $p^{(8)}(t)=\overline{d}_8$=-105,414.6... We have already checked that $p^{(j)}(5.72)<0$ for $j=0\dots 7$, so in order to conclude $p(t)>0$ for $5.72\leq t\leq 6$ it suffices to show $p^{(8)}(t)<0$ in the given interval. However, $p^{(8)}$ is constant, so $p(t)<0$ for all $t \in \RR$.  It follows that also $p(t)>0$ for all $5.72 \leq t\leq 6$.
\end{proof}

\begin{center}
\textsc{Acknowledgement}
\end{center}
The author is grateful to Professor Szil\'ard R\'ev\'esz for continuous guidance and encouragement. Also the author gratefully acknowledges the distinction and support, given to him in form of the Ames Award, for the paper \cite{Krenci}
\bigskip
\bigskip
\bigskip

\end{document}